\documentclass[a4paper,11pt]{amsart}
\usepackage{amssymb}
\usepackage[cyr]{aeguill}
\usepackage{cmap}
\usepackage{hyperxmp}
\usepackage[pdfdisplaydoctitle=true,
            colorlinks=true,
            urlcolor=blue,
            citecolor=blue,
            linkcolor=blue,
            pdfstartview=FitH,
            pdfpagemode= UseNone,
            bookmarksnumbered=true]{hyperref}

\theoremstyle{plain}
\newtheorem{thm}{Theorem}[section]
\newtheorem{cor}[thm]{Corollary}
\newtheorem{lem}[thm]{Lemma}
\newtheorem{prop}[thm]{Proposition}

\theoremstyle{definition}
\newtheorem{defn}[thm]{Definition}

\theoremstyle{remark}
\newtheorem{rem}[thm]{Remark}
\newtheorem{rems}[thm]{Remarks}
\newtheorem*{expl}{Example}

\numberwithin{equation}{section}

\newcommand{\NN}{\mathbb{N}} 
\newcommand{\RR}{\mathbb{R}} 
\newcommand{\CC}{\mathbb{C}} 

\newcommand{\Rd}{\mathbb{R}^{d}} 
\newcommand{\g}{\mathfrak{g}}
\newcommand{\gstar}{\mathfrak{g}^{*}}

\newcommand{\ph}{\varphi}
\newcommand{\id}{\mathrm{id}}

\newcommand{\DiffRd}{\mathrm{Diff}_{H^{\infty}}(\mathbb{R}^{d})} 
\newcommand{\DiffRone}{\mathrm{Diff}_{H^{\infty}}(\mathbb{R})} 
\newcommand{\DRd}[1]{\mathcal{D}^{#1}(\mathbb{R}^{d})} 

\newcommand{\HRd}[1]{H^{#1}(\mathbb{R}^{d},\mathbb{R}^{d})} 
\newcommand{\LRd}{L^{2}(\mathbb{R}^{d},\mathbb{R}^{d})} 
\newcommand{\CSCRd}{C_{c}^{\infty}(\mathbb{R}^{d},\mathbb{R}^{d})} 
\newcommand{\CS}{\mathrm{C}^{\infty}} 

\newcommand{\norm}[1]{\left\Vert#1\right\Vert}
\newcommand{\abs}[1]{\left\vert#1\right\vert}
\newcommand{\set}[1]{\left\{#1\right\}}
\newcommand{\op}[1]{\mathbf{op}\left(#1\right)}

\DeclareMathOperator{\ad}{ad} %
\DeclareMathOperator{\tr}{tr} %
\DeclareMathOperator{\grad}{grad} %
\DeclareMathOperator{\dive}{div} %

\hypersetup{
    pdftitle={Local and Global Well-posedness of the fractional order EPDiff equation on $R^{d}$}, 
    pdfauthor={M. Bauer, J. Escher and B. Kolev}, 
    pdfsubject={MSC 2010: 58D05, 35Q35}, 
    pdfkeywords={EPDiff equation; diffeomorphism groups; Sobolev metrics of fractional order}, 
    pdflang=EN 
    }

\begin{document}

\title[Well-posedness of fractional order EPDiff equations on $\Rd$]{Local and Global Well-posedness of the fractional order EPDiff equation on $\mathbb{R}^{d}$}

\author{Martin Bauer}
\address{Faculty for Mathematics, University of Vienna,  Austria}
\email{bauer.martin@univie.ac.at}

\author{Joachim Escher}
\address{Institute for Applied Mathematics, University of Hanover, D-30167 Hanover, Germany}
\email{escher@ifam.uni-hannover.de}

\author{Boris Kolev}
\address{Aix Marseille Universit\'{e}, CNRS, Centrale Marseille, I2M, UMR 7373, 13453 Marseille, France}
\email{boris.kolev@math.cnrs.fr}


\date{\today}

\begin{abstract}
Of concern is the study of fractional order Sobolev--type metrics on the group of $H^{\infty}$-diffeomorphism of $\mathbb{R}^{d}$ and on its Sobolev completions $\mathcal{D}^{q}(\mathbb{R}^{d})$. It is shown that the $H^{s}$-Sobolev metric induces a strong and smooth Riemannian metric on the Banach manifolds $\mathcal{D}^{s}(\mathbb{R}^{d})$ for $s >1 + \frac{d}{2}$. As a consequence a global well-posedness result of the corresponding geodesic equations, both on the Banach manifold $\mathcal{D}^{s}(\mathbb{R}^{d})$ and on the smooth regular Fr\'{e}chet-Lie group of all $H^{\infty}$-diffeomorphisms is obtained. In addition a local existence result for the geodesic equation for metrics of order $\frac{1}{2} \leq s <  1 + d/2$ is derived.
\end{abstract}

\maketitle

\setcounter{tocdepth}{1}
\tableofcontents

\section{Introduction}
\label{sec:introduction}

In 1966 Arnold observed in his seminal article~\cite{Arn1966}, that the incompressible Euler equations can be interpreted as the geodesic equation on the group of volume-preserving diffeomorphisms with respect to the right invariant $L^2$--metric.
Subsequently, it has been shown that there exist a geometric interpretation for many other physically relevant PDEs; e.g.:

\begin{itemize}
 \item Burgers' equation as the geodesic equation on $\operatorname{Diff}(S^1)$ with the $L^2$-metric;
 \item the Camassa--Holm equation~\cite{CH1993} on $\operatorname{Diff}(S^1)$ with the $H^1$-metric in~\cite{Kou1999};
 \item the Degasperis-Procesi equation on $\operatorname{Diff}(S^1)$ with a non-metric connection~\cite{EK2011} (it has also been described previously as an evolution equation on the space of tensor densities, see~\cite{Guh2007,LMT2010}, in the spirit of the Euler-Poincar\'{e} formalism);
 \item the KdV equation on the Virasoro--Bott group with the $L^2$-metric in~\cite{KO1987};
 \item the Hunter-Saxton equation on the homogeneous space $\operatorname{Diff}(S^1)/S^1$ with respect to the homogeneous $\dot H^{1}$-metric~\cite{Len2007, Len2008};
 \item the modified Constantin--Lax--Majda equation~\cite{CLM1985} as the geodesic equation on the homogeneous space $\operatorname{Diff}(S^1)/S^1$ with respect to the homogeneous $\dot H^{1/2}$-metric~\cite{Wun2010,EKW2012}.
\end{itemize}

These geometric interpretations as geodesic equations on infinite dimensional manifolds have been used to obtain existence and stability results for the corresponding PDEs: Ebin and Marsden~\cite{EM1970} proved local well--posedness of the geodesic initial value problem for the incompressible Euler equations. Their method is based on a extension of the metric and the geodesic spray to the Hilbert manifold of Sobolev diffeomorphisms. Using similar techniques Constantin and Kolev showed in~\cite{CK2003} that the geodesic equation of the right--invariant Sobolev metrics of integer $k\geq 1$ on $\operatorname{Diff}(S^1)$ is locally well-posed. In~\cite{EK2014} this result was extend by Escher and Kolev to the class of fractional order Sobolev metric of order $s\geq \frac{1}{2}$ and in~\cite{EK2014a} they proved global existence of geodesics on $\operatorname{Diff}(S^1)$, provided that the order
$s$ satisfies $s>\frac32$.

For diffeomorphism groups of a general (possibly) higher dimensional, compact manifold $M$ the situation has been only studied for integer order metrics: metrics of order one have been studied by Shkoller in \cite{Shk1998,Shk2000}; in~\cite{MP2010} Preston and Misiolek showed that the geodesic equation on $\operatorname{Diff}(M)$ is locally well-posed for Sobolev metrics of integer order $k\geq 1$. This result can also be found implicitly in~\cite{EM1970,BHM2011}. In a recent preprint~\cite{BV2014}, by Bruveris and Vialard, metric and geodesic completeness on the Banach manifold $\mathcal D^s(M)$ have been studied provided that the metric is smooth and strong. Using a result of~\cite{EM1970}, see also~\cite{MP2010}, this is true for the class of Sobolev metrics of sufficiently high integer order. This is in correspondence with the results of Mumford and Michor in~\cite{MM2013}, where global well-posedness on the group of $H^{\infty}$-diffeomorphisms of $\mathbb R^{d}$ for integer order metrics is proven, see also~\cite{TY2005}. However, so far both local and global well-posedness was left open for fractional order Sobolev metrics. This article serves as a contribution towards this goal, as it discusses local and global well-posedness for fractional order Sobolev metrics on the diffeomorphism group of $\mathbb R^{d}$. Our main result is the following:

\begin{thm}
Let $M$ be either the $d$-dimensional Torus $\mathbb{T}^{d}$ or the Euclidean space
$\Rd$.
Let $G^s$ be the fractional order Sobolev type metric of order $s$ on the group of $H^{\infty}$-diffeomorphisms $\operatorname{Diff}(M)$.
\begin{enumerate}
 \item Let $s\geq \frac 12$. Then, given any $(\varphi_{0},v_{0})\in T\mathrm{Diff}(M)$, there exists a unique non-extendable geodesic $(\varphi, v) \in C^{\infty}(J,T\mathrm{Diff}(M))$ defined on the maximal time interval $J$, which is open and contains $0$.
\item All geodesics exist globally, provided $s>\frac d2+1$.
\end{enumerate}
\end{thm}

\begin{rem}
In all of the article we will only treat the case $M=\Rd$. The proofs work without any substantial change for $M=\mathbb{T}^{d}$.
\end{rem}

\paragraph{\textbf{Outline of the article.}}
In the following we give a short overview of the structure of the article: In Section~\ref{sec:Dq}, we introduce the diffeomorphism groups studied in this article and review basic results on their manifold and group structure. In Section~\ref{sec:EPDiff}, we define the class of right-invariant metrics -- and in particular of right-invariant Sobolev metrics -- and present their geodesic equation. The metrics in this article are defined via Fourier multipliers of class $\mathcal S^r$, which is the content of Section~\ref{sec:Fourier-multipliers}. In Section~\ref{sec:smoothness}, we show that these metrics induce a smooth metric on the Sobolev completions $\DRd{q}$, for sufficiently high $q$ and $r$. This is then used to prove local (Section~\ref{sec:well-posedness}) and global (Section~\ref{sec:global-well-posedness}) well-posedness. Some facts on Fourier multipliers which are needed in our analysis are collected in the Appendices A and B. Finally, the technically most involved computations for the
symbols of higher derivatives of the conjugation are postponed to Appendix C.

\section{The groups of smooth and Sobolev diffeomorphisms}
\label{sec:Dq}

In this paper, we will be mainly interested in the diffeomorphism group
\begin{equation}\label{eq:Diff_infinity}
    \DiffRd := \set{\id + f;\; f \in \HRd{\infty} \; \text{and} \; \operatorname{det}(\id + df)> 0}\,.
\end{equation}
Here $ \HRd{\infty}$ denotes the space of $\Rd$-valued $H^{\infty}$-functions on $\Rd$, i.e.,
\begin{equation*}
    \HRd{\infty} := \bigcap_{q \ge 0} \HRd{q}\,,
\end{equation*}
where $\HRd{q}$ denotes the ($\Rd$-valued) Sobolev space on $\Rd$, see Section \ref{subsec:Hq}.
In addition we will need the Hilbert approximations $\DRd{q}$ of $\DiffRd$:
\begin{equation*}
    \DRd{q} := \set{\id + f;\; f \in \HRd{q} \; \text{and} \; \operatorname{det}(\id + df)> 0}\,,
\end{equation*}
defined for $q>\frac{d}2+1$. In this section we will recall the basic definitions and results to rigorously define these groups. For a more detailed treatment of the groups $\DRd{q}$ we refer to the monograph~\cite{IKT2013} and for the group $\DiffRd$ to~\cite{HD2010,MM2013a}.

\subsection{The Sobolev space $\HRd{q}$.}
\label{subsec:Hq}

The Fourier transform $\mathcal{F}$ on $\RR^{d}$ will be defined as
\begin{equation*}
  (\mathcal{F} f)(\xi) = \int_{\Rd} e^{-2i\pi \langle  x,\xi\rangle} f(x) \, dx
\end{equation*}
where $\xi$ is the independent variable in the frequency domain. With this convention, its inverse $\mathcal{F}^{-1}$ is given by:
\begin{equation*}
  (\mathcal{F}^{-1} g)(x) = \int_{\Rd} e^{2i\pi \langle  x,\xi\rangle} g(\xi) \, d\xi \, .
\end{equation*}

For $q\in \RR^+$ the Sobolev $H^q$-norm of an $\Rd$-valued function $f$ on $\Rd$ is
defined by
\begin{equation*}
  \norm{f}_{H^{q}}^{2} := \norm{(1+|\xi|^{2})^{\frac{q}2} \mathcal{F} f}_{L^{2}}^{2}\, .
\end{equation*}
Using this norm we obtain the Sobolev spaces of -- possibly non-integral -- order $q$:
\begin{equation*}
  \HRd{q} = \set{ f \in L^{2}(\Rd,\Rd): \norm{f}_{H^{q}} < \infty }\;.
\end{equation*}
These spaces are also known under the name \emph{Liouville spaces} or \emph{Bessel potential
spaces}. To make a connection with other families of function spaces, we note that the
spaces $\HRd{q}$ coincide with
\begin{equation*}
  \HRd{q} = B^q_{22}(\Rd,\Rd) = F^q_{22}(\Rd,\Rd)
\end{equation*}
the \emph{Besov spaces} $B^q_{22}(\Rd,\Rd)$ and with the spaces of \emph{Triebel--Lizorkin type}
$F^q_{22}(\Rd,\Rd)$. Definitions of all these spaces and an introduction to the
general theory of function spaces can be found in~\cite{Tri1983,Tri1992}.

In the following lemma, we collect two important properties of these spaces that we will use throughout this article.

\begin{lem}\label{lem:sobolev_embedding}
Let $q\geq 0$ be any non-negative real number. Then we have:
\begin{itemize}
 \item The space of smooth and compactly supported functions $C^{\infty}_c(\Rd,\Rd)$ is a dense subset of $\HRd{q}$.
 \item If $q > d/2$ then the space $H^{q+r}(\Rd,\Rd)$ can be embedded into the space $C^r_0(\Rd,\Rd)$ of all $C^r$-functions vanishing at infinity for any integer $r$.
 \item If $q > d/2$ and $p\geq1$ with $p \le q$ then pointwise multiplication extends to a
bounded bilinear mapping
\begin{equation*}
  \HRd{q} \times \HRd{p} \to \HRd{p}.
\end{equation*}
\end{itemize}
\end{lem}
\begin{proof}
The proof of this Lemma can be found in~\cite[Section~2]{IKT2013}.
\end{proof}

\subsection{The Hilbert manifold $\DRd{q}$.}

We are now able to rigorously define the Hilbert manifold $\DRd{q}$.
We will follow the presentation of~\cite{IKT2013}.
Given any $q>\frac{d}{2}+1$, we let
\begin{equation*}
 \DRd{q} = \set{ \ph\in \operatorname{Diff}^1_+(\Rd): \ph- \id \in \HRd{q}}\,.
\end{equation*}
Here $\operatorname{Diff}^1_+(\Rd)$ denotes the set of orientation-preserving $C^1$-diffeomorphisms on $\Rd$.
In~\cite{IKT2013} it has been shown, that equivalent definitions of this group are given by
\begin{align*}
 \DRd{q} &= \set{ \ph: \ph \text{ is bijective},\, \ph,\ph^{-1} \in \id+ \HRd{q}}
 \\
 &=\set{\ph= \id+f:  f\in \HRd{q},\, \operatorname{det}(\id+df)>0}\,.
\end{align*}
Note, that for a compact manifold $M$, this has already been proven in~\cite{Ebi1970}.
Furthermore the mapping
\begin{equation*}
\DRd{q} \rightarrow \HRd{q},\qquad \ph\mapsto \ph-\id
\end{equation*}
provides a smooth chart for $\DRd{q}$ for any $q>\frac{d}2+1$.
Thus we have given $\DRd{q}$ the structure of a smooth Hilbert manifold.

\begin{rem}
Note, that the tangent bundle $T\DRd{q}$ is a trivial bundle
\begin{equation*}
 T\DRd{q}\cong\DRd{q}\times \HRd{q}\,,
\end{equation*}
because $\DRd{q}$ is an open subset of the Hilbert space $\HRd{q}$.
\end{rem}

\begin{rem}\label{Rem:Jacobianbounded}
By the definition of $\DRd{q}$ we can bound the norm of the Jacobian determinant of any diffeomorphism $\varphi \in \DRd{q}$ from below, i.e.,
\begin{equation*}
 \norm{\det(d\varphi)}_{\infty}> C_{1}(\varphi), \qquad \forall \varphi \in \DRd{q}\,.
\end{equation*}
On the other hand, we also have a similar bound from above, i.e.,
\begin{equation*}
 \norm{\det(d\varphi)}_{\infty}< C_{2}(\varphi), \qquad \forall \varphi \in \DRd{q}\,.
\end{equation*}
\end{rem}

\subsection{The topological group $\DRd{q}$.}

The following lemma shows that the manifold $\DRd{q}$ is in addition a topological group but not a Lie group, as
composition and inversion in $\DRd{q}$ are continuous but not smooth.
More precisely we have:

\begin{lem}[Proposition 2.6 in~\cite{IKT2013}]
 Given any real number $q$ with $q>\frac{d}{2}+1$, the Hilbert manifold $\DRd{q}$ is a topological group, but never a Lie group;
 the mappings
 \begin{align*}
  &\mu: \DRd{q} \times\DRd{q} \rightarrow \DRd{q}, &(\ph,\psi)\mapsto \ph\circ\psi\\
  &\nu: \DRd{q} \rightarrow \DRd{q}, &\ph\mapsto \ph^{-1}\, .
 \end{align*}
are continuous, but not smooth.
\end{lem}

We will also need the following result concerning the right action of $\DRd{q}$ on Sobolev functions.

\begin{lem}[Lemma 2.7 in~\cite{IKT2013}]\label{lem:composition}
Given any two real numbers $q,p$ with $q>\frac{d}{2}+1$ and $q\geq p\geq 0$, the mapping
 \begin{equation*}
  \mu^{p}: \HRd{p} \times\DRd{q} \rightarrow \HRd{p}, \qquad (u,\ph) \mapsto u \circ \ph.
 \end{equation*}
is continuous. Moreover, the mapping
\begin{equation*}
  R_{\ph}: u \mapsto u \circ \ph
\end{equation*}
is \emph{locally bounded}. More precisely, given $C_{1}, C_{2}>0$, there exists a constant $C=C(p,C_{1},C_{2})$ such that
\begin{equation*}
  \norm{R_{\ph}}_{\mathcal{L}(H^{p},H^{p})} \leq C ,
\end{equation*}
for all $\ph\in \DRd{q}$ with
\begin{equation*}
  \norm{\ph-\id}_{H^{q}} < C_{1} \quad \text{and} \quad \underset{x\in\Rd}{\inf}\left( \operatorname{det}(d\ph(x))\right) > C_{2},
\end{equation*}
\end{lem}

Note, that the first part of this result is the $\Rd$ version of Corollary B.3 in~\cite{EK2014}. In the one-dimensional case
more exact bounds for the composition have been derived in Corollary B.2 in~\cite{EK2014}.

\subsection{A complete metric structure on $\DRd{q}$.}

In this section we will generalize a complete metric on $ \mathcal{D}^{q}(S^1)$, as introduced in~\cite[Section 3]{EK2014a}, to the situation studied in this article.

\begin{defn}
Given any $q>\frac{d}2+1$, we define the distance function
\begin{equation*}
d_q(\ph_{1},\ph_{2}) =\norm{\ph_{1}-\ph_{2}}_{H^{q}}+\norm{\det(d\ph_{1})^{-1}-\det(d\ph_{2})^{-1}}_{\infty}\,.
\end{equation*}
\end{defn}

\begin{thm}\label{thm:complete_metric}
Given any $q>\frac{d}{2}+1$, the space $\left(\DRd{q},d_q \right)$ is a complete metric space and the metric topology is equivalent to the Hilbert manifold topology.
\end{thm}
The main ingredient of this proof is a bound on the Jacobian determinant on bounded subsets of $\left(\DRd{q},d_q\right)$.
\begin{lem}
Let $q>\frac{d}{2}+1$. Given any (non-empty) bounded subset $B$ of $\left(\DRd{q},d_q\right)$, we have
\begin{equation*}
\underset{\ph\in B }{\operatorname{inf}} \left(\underset{y\in \Rd }{\operatorname{inf}} \operatorname{det}(d\ph(y)) \right)>0\,.
\end{equation*}
\end{lem}
\begin{proof}
We consider some fixed $\ph_0$ in $B$ and let
\begin{equation*}
\varepsilon := \frac{1}{\operatorname{diam}(B)+ \|1/\operatorname{det}(d\ph)\|_{\infty}} \,.
\end{equation*}
Since $B$ is bounded and $\norm{1/\det(d\varphi_{0})}_{\infty} < +\infty$ -- c.f. Remark \ref{Rem:Jacobianbounded} -- we have $\varepsilon>0$.
Suppose that
\begin{equation*}
\underset{\ph\in B }{\operatorname{inf}} \left(\underset{y\in \Rd }{\operatorname{inf}} \operatorname{det}(d\ph(y)) \right)=0\,.
\end{equation*}
Then there exists a $\ph_{1}$ such that
$\underset{y\in \Rd }{\operatorname{inf}} \operatorname{det}(d\ph_{1}(y))<\varepsilon$.
But then we have
\begin{align*}
d_q(\ph_{1},\ph_{2}) &\geq   \left\| \operatorname{det}(d\ph_{1})^{-1}-\operatorname{det}(d\ph_{2})^{-1} \right\|_{\infty}\\&\geq  \left\| \operatorname{det}(d\ph_{1})^{-1}\right\|_{\infty}-\left\| \operatorname{det}(d\ph_{2})^{-1} \right\|_{\infty}
\\&>\frac1{\varepsilon}+\operatorname{diam}(B)-\frac1{\varepsilon}\\&=\operatorname{diam}(B)\,.
\end{align*}
This yields a contradiction to $\ph_{1}$ being an element of $B$.
\end{proof}
Using this lemma, the proof of Theorem~\ref{thm:complete_metric} is \textit{verbatim} the same as in the one-dimensional periodic situation, see~\cite[Section 3]{EK2014a}.

\subsection{The Lie group $\DiffRd$.}

The following result concerning the Lie-group structure of $\DiffRd$ has been first shown in~\cite{HD2010a}.

\begin{thm}[Hermas \& Djebali~\cite{HD2010a}]
The space $\DiffRd$ as defined in \eqref{eq:Diff_infinity} is a \emph{regular Fr\'{e}chet-Lie group} with Lie-Algebra
$\mathfrak X_{\infty}(\RR^{d})= \HRd{\infty}$ the space of $H^{\infty}$-vector fields.
\end{thm}

The Lie-group structures of the related groups $\operatorname{Diff}_{c}(\RR^{d})$, $\operatorname{Diff}_{\mathcal S}(\RR^{d})$ and $\operatorname{Diff}_{\mathcal B}(\RR^{d})$ have been studied in~\cite{MM2013}. It has been shown, that $\operatorname{Diff}_{c}(\RR^{d})$ is a simple group. However, this result does not carry over to the larger group $\DiffRd$, as both $\operatorname{Diff}_{c}(\RR^{d})$ and $\operatorname{Diff}_{\mathcal S}(\RR^{d})$ are normal subgroups of it.

\section{The EPDiff equation}
\label{sec:EPDiff}

\subsection{Right-invariant metrics on Lie groups.}

A \emph{right-invariant} Riemannian metric on a  Lie group is defined by its value at the unit element $e$ of the
group, that is, by a inner product on the Lie algebra $\g$ of $G$. For historical reasons going back to Euler~\cite{Eul1765a}, this inner product is usually represented by a \emph{symmetric}\footnote{$A$ is symmetric if $(Au,v)=(Av,u)$ for all $u,v \in \g$, where the
round brackets stand for the dual pairing of elements of $\g$ and its dual space $\gstar$.} linear operator
\begin{equation*}
    A: \g \to \gstar ,
\end{equation*}
called the \emph{inertia operator}. Given a curve $g(t)$ on $G$, the \emph{Eulerian velocity} (also called the \emph{(right) logarithmic derivative}) is defined as
\begin{equation*}
  u := TR_{g^{-1}}. g_{t},
\end{equation*}
where $R_{g}$ stands for right translations on $G$. With these notations, the \emph{geodesic equations} can be written as
\begin{equation*}
  g_{t} = TR_{g}. u, \qquad u_{t} = -B(u,u),
\end{equation*}
where
\begin{equation*}
  B(u,v)= \frac12 \left(\ad(u)^{\top} v + \ad(v)^{\top} u \right), \quad u,v \in \g
\end{equation*}
and $\ad(u)^{\top}$ is the adjoint of the operator $\ad(u)$, with respect to the inner product on $\g$. The first order equation on $u$ is called the \emph{Arnold-Euler equation} and the bilinear operator $B$ is called the \emph{Arnold operator}.

As noted by Arnold~\cite{Arn1966}, the theory can be extended to diffeomorphism groups (and more generally \emph{Fr\'{e}chet-Lie groups}) but with two restrictions:
\begin{enumerate}
  \item The metric does not induce, generally, an isomorphism between the tangent space $T_{g}G$ and its dual;
  \item The \emph{geodesic spray} may not exist.
\end{enumerate}

\subsection{Right-invariant metrics on $\DiffRd$.}

To define a right invariant metric on $\DiffRd$, it suffices to prescribe an inner product on the Lie algebra $\HRd{\infty}$. Therefore let $A: \HRd{\infty} \to \HRd{\infty}$ be a $L^{2}$-symmetric, positive definite, continuous linear operator. Then $A$ induces a positive inner product on the Fr\'{e}chet space $\HRd{\infty}$, given by
\begin{equation*}
  \langle u_{1},u_{2}\rangle := \int_{\Rd}  Au_{1}\cdot u_{2} \, dx ,
\end{equation*}
where $\cdot$ denotes the Euclidean product on $\Rd$. We can use this inner product to define an inner product on each tangent space $T_{\ph}\DiffRd$. For any $\ph\in \DiffRd$ and $v_{1},v_{2}\in T_{\ph}\DiffRd$ it is defined by
\begin{equation*}
  G_{\ph}(v_{1},v_{2}) = \langle v_{1}\circ\ph^{-1}, v_{2}\circ\ph^{-1}\rangle = \int_{\Rd}  A (v_{1}\circ\ph^{-1})\cdot  (v_{2}\circ\ph^{-1}) dx\,.
\end{equation*}
Using the notation $A_{\ph}=R_{\ph}\circ A\circ R_{\ph^{-1}}$ we can rewrite -- via a change of variable -- the above equation to obtain
\begin{equation*}
  G_{\ph}(v_{1},v_{2}) = \int_{\Rd}  \left( A_\ph v_{1} \cdot v_{2} \right) J_{\ph}\, dx\,,
\end{equation*}
where $J_{\ph}$ denotes the Jacobian determinant of the coordinate change $\ph$. The operator $A$ is also called the \emph{inertia operator} of the metric $G$.

In the next theorem we will calculate the corresponding Euler equation, assuming that the inertia operator is invertible. This equation is known in the literature as the \emph{EPDiff equation} (Euler-Poincar\'{e} equation on the diffeomorphism group).

\begin{thm}[EPDiff equation]\label{thm:EPDiff}
If the inertia operator
\begin{equation*}
  A: \HRd{\infty} \to \HRd{\infty}
\end{equation*}
is invertible, then the Arnold operator $B$ exists and the corresponding Euler equation is given by
\begin{equation}\label{eq:EPDiff}
  m_{t} + \nabla_{u}m + \left(\nabla u\right)^{t}m + (\dive u) m  = 0, \quad m := Au.
\end{equation}
\end{thm}

\begin{proof}
Let $\nabla$ be the canonical covariant derivative on $\RR^{d}$. For $u,v,w\in \CSCRd$ we calculate
\begin{align*}
  \langle u, \ad(v)w \rangle & = -\int_{\Rd} Au \cdot [v,w] \,dx
  \\
  & = -\int_{\Rd} Au \cdot \left( \nabla_{v}w - \nabla_{w}v \right) \,dx
  \\
  & = \int_{\Rd} \left[ - \grad (Au \cdot w) \cdot v + \left( \nabla_{v}Au \right) \cdot w + (\nabla v )^{t} Au \cdot w \right] \,dx
  \\
  & = \int_{\Rd} \left[ (Au \cdot w) \dive v + \left( \nabla_{v}Au \right) \cdot w + (\nabla v )^{t} Au \cdot w \right] \,dx
  \\
  & = \int_{\Rd}\left( \nabla_{v}Au + (\nabla v )^{t} Au + (\dive v) Au \right) \cdot w \,dx.
\end{align*}
Since $\CSCRd$ is dense in $\HRd{\infty}$ and both sides of the equation are continuous, the same relation holds for $u,v,w\in \HRd{\infty}$. We deduce therefore that
\begin{equation*}
  \ad(v)^{\top} u = A^{-1}\left( \nabla_{v}Au + (\nabla v )^{t} Au + (\dive v) Au \right),
\end{equation*}
and the Euler equation~\eqref{eq:EPDiff} follows.
\end{proof}

\begin{rem}
Note that the hypothesis that the inertia operator is invertible is only a sufficient condition for the existence of the operator $B$. A weaker hypothesis is that the symmetric part of the bilinear operator $\ad(v)^{\top} u$ is defined, which is the case if
\begin{equation*}
  \nabla_{v}Au + (\nabla v )^{t} Au + (\dive v) Au + \nabla_{u}Av + (\nabla u )^{t} Av + (\dive u) Av
\end{equation*}
belongs to the range of $A$ for all $u,v \in \HRd{\infty}$.
\end{rem}

The \emph{geodesic spray}, if it exists, is defined as the Hamiltonian vector field of the energy function
\begin{equation*}
  \mathcal{E}(\varphi, v) := \frac{1}{2} G_{\ph}(v,v), \quad (\varphi, v) \in T\DiffRd
\end{equation*}
for the (weak) symplectic structure on $T\DiffRd$ induced by the metric. As pointed out by Arnold~\cite{Arn1966}, the geodesic spray exists as soon as the Arnold operator $B$ exists.

\begin{thm}
  Given a right-invariant metric on $T\DiffRd$, the geodesic spray exists if and only if the Arnold operator $B$ exists.
\end{thm}

\begin{proof}[Sketch of proof]
If the spray exists, it is uniquely defined and right-invariant. It can be checked that, if the operator $B$ exists, then the spray is given by
\begin{equation*}
  F(\varphi,v) := (\varphi, v, v, S_{\varphi}(v)),
\end{equation*}
where $S_{\varphi}(v) := R_{\varphi} \circ S \circ R_{\varphi^{-1}}(v)$ and
\begin{equation*}
  S(u) := \nabla_{u}u - B(u,u).
\end{equation*}
Conversely, if the spray exists, then a solution $B$ of the equation
\begin{equation*}
  \langle B(u,v), w \rangle = \frac{1}{2} \left( \langle v, \ad(u)w \rangle + \langle u, \ad(v)w \rangle \right)
\end{equation*}
is given by the polarization of the quadratic operator
\begin{equation*}
  B(u,u) := \nabla_{u}u - S_{\id}(u). \qedhere
\end{equation*}
\end{proof}

\begin{expl}
In the following we present an example of a right invariant metric, where the Arnold operator $B$ -- and thus the geodesic equation -- does not exist. For this we consider the homogeneous $\dot H^1$-metric on the diffeomorphism group of the real line $\DiffRone$:
\begin{equation*}
  \langle u, v \rangle_{\dot H^1}:=\int_{\RR} u_{x} \cdot v_{x} \,dx\,,\quad \text{ for } u,v \in H^{\infty}(\RR,\RR)\,.
\end{equation*}
This defines a Riemannian metric on $\DiffRone$, since there are no constant vector fields in the Lie-Algebra $H^{\infty}(\RR)$.

One can formally derive the equation for the Arnold bilinear operator to obtain:
\begin{align*}
B(u,v) = \int_{-\infty}^x u_{x} \cdot  v_{x}\, dy + (u \cdot  v)_{x}\,.
\end{align*}
Thus, if $B(u,v)$ is an element of the Lie algebra $H^{\infty}(\RR)$, then
\begin{equation*}
   \int_{-\infty}^{\infty} u_{x} \cdot  v_{x} =0\,.
\end{equation*}
For $u=v$ this would imply that $u_{x}=0$. Therefore the Arnold bilinear operator does not exist for all $u,v\in H^{\infty}(\RR)$.

In the article~\cite{BBM2014} possible extensions of the group $\DiffRone$ have been discussed in order to guarantee the existence of $B$. However, this is only possible if one uses either the group of compactly supported or rapidly decreasing diffeomorphisms, but not for $H^{\infty}$-diffeomorphisms that are treated in this article.

Note that a similar phenomenon would also occur for any \emph{homogeneous metric} of order $s>0$ on $\DiffRd$.

This is in contrast to the periodic case. To study the $\dot H^1$ metric in the periodic case one has to pass to the homogeneous space
$\operatorname{Diff}(S^1)/S^1$ of diffeomorphisms modulo rotations. Then the geodesic equation exists and is given by the Hunter-Saxton equation. Furthermore, the induced geometry of the $\dot H^1$-metric is an open subset of an infinite dimensional sphere, cf.~\cite{Len2007, Len2008}.
\end{expl}

\subsection{Sobolev metrics on $\DiffRd$.}

A class of metrics of particular importance is given by the family of (fractional order) Sobolev metrics.
As described in the previous section we only need to define the metric on the Lie algebra $\HRd{\infty}$ and obtain a metric on all of
$\DiffRd$ via right translation. We will first consider the Sobolev metric of integer order $k\in \NN$. Therefore we define the inner product:
\begin{equation}\label{eq:Hk-inner-product}
  \langle u_{1},u_{2}\rangle_{H^{k}} : = \sum_{j=0}^{k} \int \nabla^{j}u_{1} \cdot \nabla^{j}u_{2} \, dx\,,
\end{equation}
where $u_{1},u_{2}\in \HRd{\infty}$ and $\cdot$ is the natural extension of the euclidean inner product on $\Rd$ to higher tensors.
The \emph{Laplacian} of a vector field $u$ will be defined as
\begin{equation*}
  \Delta u := \tr \nabla^{2}u = \sum_{i,j}\nabla^{2}_{\partial_{i},\partial_{j}}u.
\end{equation*}
Due to the identity
\begin{equation*}
  \dive (\nabla u_{1} \cdot u_{2}) = \Delta u_{1} \cdot u_{2} + \nabla u_{1} \cdot \nabla u_{2}\,,
\end{equation*}
the $H^{1}$ inner product can be rewritten as
\begin{equation*}
  \langle u_{1},u_{2}\rangle_{H^{1}} =  \int (1 - \Delta)u_{1} \cdot u_{2} \, dx\, .
\end{equation*}
This suggests to introduce a \emph{modified but norm-equivalent} version of the $H^{k}$--inner product~\eqref{eq:Hk-inner-product}, namely
\begin{equation*}
  \langle u_{1},u_{2}\rangle_{H^{k}} : = \int (1 - \Delta)^{k}u_{1} \cdot u_{2} \, dx\, .
\end{equation*}
The fact that both norms are equivalent is based on the inequality
\begin{equation*}
  \left(1 + \abs{\xi}^{2}\right)^{k} \lesssim \sum_{j=0}^{k} \abs{\xi}^{2j} \lesssim \left(1 + \abs{\xi}^{2}\right)^{k}.
\end{equation*}

Similarly as in Section~\ref{subsec:Hq}, this definition can be extended via the Fourier transform to obtain Sobolev metrics of non-integer order $s\ge 0$:
\begin{equation}\label{equation:Sob_norm}
  \langle u_{1},u_{2}\rangle_{H^{s}} : = \int \left(1 - \Delta \right)^{s}u_{1} \cdot u_{2} \, dx = \int \left(1 + \abs{\xi}^{2}\right)^{s} \hat{u}_{1} \cdot \hat{u}_{2} \, dx,
\end{equation}
where $\cdot$ is the Hermitian inner product on $\CC^{d}$.

In the spirit of the previous section, it is natural to consider the corresponding inertia operator $A$ of the norm \eqref{equation:Sob_norm}:
\begin{equation*}
 A= \Lambda^{2s} := \left(1 + \abs{\xi}^{2}\right)^{s}(D).
\end{equation*}
This operator belongs to the family of \emph{Fourier multipliers of class $\mathcal{S}^{2s}$}, which is defined in following section.

\section{Fourier multipliers of class $\mathcal{S}^{r}$}
\label{sec:Fourier-multipliers}

It is a well-known fact (see for instance~\cite[Appendix A]{EM1970}) that, given any \emph{differential operator} $A$ of order $r$ with \emph{smooth coefficients}, the mapping
\begin{equation*}
	\varphi \mapsto A_{\varphi}:= R_{\varphi} \circ A \circ R_{\varphi^{-1}} , \quad \DRd{q} \to \mathcal{L}(\HRd{q},\HRd{q-r})
\end{equation*}
is smooth for $q > 1 + d/2$ and $q-r \ge 0$. It is the aim of this section to extend this result to a larger class of operators, the so-called \emph{Fourier multipliers of class $\mathcal{S}^{r}$}, which are defined below.

Let $A$ be a differential operator with constant coefficients, then
\begin{equation*}
  \widehat{(Au)}(\xi) = a(\xi) \hat{u}(\xi), \qquad u \in \HRd{\infty},
\end{equation*}
where $a: \Rd \to \mathcal{L}(\CC^{d})$ is a polynomial function. This observation suggests to define, for a more general function $a: \Rd \to \mathcal{L}(\CC^{d})$, a linear operator by the following formula
\begin{equation*}
  a(D)u:= \mathcal{F}^{-1}(a\,\hat u),\quad u \in \HRd{\infty},
\end{equation*}
where $\mathcal{F}$ is the Fourier transform on $\Rd$.
Such an operator, $a(D)$ (also noted $\op{a(\xi)}$), is called a \emph{Fourier multiplier with symbol $a$}.
Of course, some regularity conditions are required on the symbol $a$ to insure that the operator is well-defined.
We redirect to Appendix~\ref{sec:translation-invariant-operators} for a detailed discussion on these conditions.
In the following, we will restrict ourselves to a class of symbols for which this operation is well-defined on $\HRd{\infty}$
and leads to operators with nice properties.

\begin{defn}
Given $r \in \RR$, a Fourier multiplier $a(D)$ is of class $\mathcal{S}^{r}$ iff $a\in C^\infty(\Rd,\mathcal{L}(\CC^{d}))$ and satisfies moreover the following condition:
\begin{equation*}
  \norm{\partial^{\alpha}a(\xi)} \lesssim \left( 1 + \abs{\xi}^{2}\right)^{(r-\abs{\alpha})/2},
\end{equation*}
for each $\alpha \in \NN^{d}$, where $\abs{\alpha} := \alpha_{1} + \dotsb + \alpha_{d}$.
\end{defn}

\begin{expl}
Any linear differential operator of order $r$ with constant coefficients is in this class. Furthermore, $\op{(1+\abs{\xi}^{2})^{r/2}}$ belongs to this class.
\end{expl}

\begin{rem}
Note that a Fourier multiplier $a(D)$ of class $\mathcal{S}^{r}$ extends to a \emph{bounded linear operator}
\begin{equation*}
  \HRd{q} \to \HRd{q-r}
\end{equation*}
for any $q\in\RR$, where $\HRd{q}$ is defined as the \emph{Banach dual space} of $\HRd{-q}$, if $q <0$.
\end{rem}

In this paper, we are mainly interested by inner products of the form
\begin{equation*}
  \langle u_{1},u_{2}\rangle := \int_{\Rd}  Au_{1}\cdot u_{2} \, dx ,
\end{equation*}
on $\HRd{\infty}$, where $A = a(D)$ is a Fourier multiplier of class $\mathcal{S}^{r}$. A necessary and sufficient condition for such an inner product to be \emph{$L^{2}$-symmetric and positive definite} is that the symbol $a$ is \emph{Hermitian and positive definite} almost everywhere. Moreover we will require an \emph{ellipticity condition} on $A$ in order to prove the existence and smoothness of the spray of the week right-invariant metric generated by the inertia operator $a(D)$. A detailed discussion on ellipticity for Fourier multipliers can be found in Appendix~\ref{sec:elliptic-Fourier-multipliers}. For our purpose, we will adopt the following definition.

\begin{defn}
A Fourier multiplier $a(D)$ in the class $\mathcal{S}^{r}$ is called \emph{elliptic} if $a(\xi)\in\mathcal{GL}(\CC^{d})$ for all $\xi\in\Rd$ and
\begin{equation*}
  \norm{[a(\xi)]^{-1}} \lesssim \left( 1 + \abs{\xi}^{2}\right)^{-r/2}, \qquad \forall \xi \in \Rd.
\end{equation*}
\end{defn}

\begin{rem}
A sufficient (but not necessary) condition for ellipticity is the following
\begin{equation*}
  \abs{\det a(\xi)} \gtrsim \left( 1 + \abs{\xi}^{2}\right)^{-dr/2},\qquad \forall \xi \in \Rd.
\end{equation*}
\end{rem}

\begin{rem}
An elliptic Fourier multiplier of class $\mathcal{S}^{r}$ induces a \emph{bounded isomorphism} between $\HRd{q}$ and $\HRd{q-r}$ for all $q\in \RR$.
\end{rem}

We summarize our considerations by introducing the following class of inertia operators which will be denoted $\mathcal{E}^{r}$.

\begin{defn}\label{def:class-Er}
An operator $A\in \mathcal{L}(\HRd{\infty})$ is in the class $\mathcal{E}^{r}$ iff the following conditions are satisfied:
\begin{enumerate}
  \item $A = a(D)$ is a Fourier multiplier of class $\mathcal{S}^{r}$;
  \item $A = a(D)$ is elliptic;
  \item $a(\xi)$ is Hermitian and positive definite for all $\xi \in \Rd$.
\end{enumerate}
\end{defn}

Since a positive definite Hermitian matrix has a unique positive square root which depends smoothly on its coefficients, we can define formally the square root $B:= \op{a(\xi)^{1/2}}$ of an operator $A$ in the class $\mathcal{E}^{r}$.

\begin{lem}\label{lem:square-root}
The positive square root $B$ of an operator $A$ in the class $\mathcal{E}^{r}$ belongs to the class $\mathcal{E}^{r/2}$.
\end{lem}

The proof of this result relies on the following elementary lemma.

\begin{lem}\label{lem:linear-algebra}
Let $a,b,x \in \mathrm{M}_{d}(\CC)$ be three matrices, where $b$ is assumed to be \emph{Hermitian and positive definite}. Suppose in addition that
\begin{equation*}
  bx + xb = a.
\end{equation*}
Then
\begin{equation*}
  \norm{x} \le \sqrt{\frac{d}{2}} \norm{b^{-1}} \norm{a},
\end{equation*}
where $\norm{\cdot}$ denotes the Frobenius norm, i.e. $\norm{x} := \sqrt{\tr x x^{*}}$.
\end{lem}

\begin{proof}
Since conjugation of the unitary group $\mathbb{U}(d)$ on $\mathrm{M}_{d}(\CC)$ is an isometry for the Frobenius norm, we can assume that $b = \mathrm{diag}(\lambda_{1},\dotsc , \lambda_{d})$. Thus $x_{i}^{j} = (\lambda_{i}+\lambda_{j})^{-1}a_{i}^{j}$ and
\begin{equation*}
  \norm{x}^{2} = \sum_{i,j} \abs{x_{i}^{j}}^{2} \le \sum_{i,j} \frac{1}{2\lambda_{i}\lambda_{j}}\abs{a_{i}^{j}}^{2} \le \frac{1}{2} (\tr b^{-1})^{2}\norm{a}^{2} \le \frac{d}{2} \norm{b^{-1}}^{2}\norm{a}^{2},
\end{equation*}
which yields the result.
\end{proof}

\begin{proof}[Proof of Lemma~\ref{lem:square-root}]
Given $\xi \in \Rd$, let $b(\xi)$ be the positive square root of $a(\xi)$. Using the Frobenius norm, we get
\begin{equation*}
  \norm{b(\xi)}^{2} = \tr b(\xi)^{2} = \tr a(\xi) \le \sqrt{d}\norm{a(\xi)}.
\end{equation*}
Thus,
\begin{equation*}
  \norm{b(\xi)} \lesssim \left( 1 + \abs{\xi}^{2}\right)^{r/4},
\end{equation*}
and similarly
\begin{equation*}
  \norm{[b(\xi)]^{-1}} \lesssim \left( 1 + \abs{\xi}^{2}\right)^{-r/4}.
\end{equation*}
We will now show by induction on $\abs{\alpha}$ that
\begin{equation}\label{eq:Sr-inductive-hyp}
  \norm{\partial^{\alpha}b(\xi)} \lesssim \left( 1 + \abs{\xi}^{2}\right)^{(r/2-\abs{\alpha})/2},
\end{equation}
for all $\alpha \in \NN^{d}$. For $\abs{\alpha} = 1$, we have
\begin{equation*}
  \partial_{i}a(\xi) = b(\xi) \left(\partial_{i}b(\xi)\right) + \left(\partial_{i}b(\xi)\right) b(\xi)
\end{equation*}
and using Lemma~\ref{lem:linear-algebra}, we conclude that
\begin{equation*}
  \norm{\partial_{i}b(\xi)} \lesssim \norm{[b(\xi)]^{-1}} \norm{\partial_{i}a(\xi)} \lesssim \left( 1 + \abs{\xi}^{2}\right)^{(r/2-1)/2}.
\end{equation*}
Suppose now that~\eqref{eq:Sr-inductive-hyp} is true for all $\abs{\alpha} \le n$ and let $\abs{\alpha} = n+1$. We have
\begin{equation*}
  \partial^{\alpha}a(\xi) = b(\xi)(\partial^{\alpha}b(\xi)) + (\partial^{\alpha}b(\xi))b(\xi) + \sum_{\abs{\beta_{1}} + \abs{\beta_{2}} = n+1} \partial^{\beta_{1}}b(\xi)\,\partial^{\beta_{2}}b(\xi),
\end{equation*}
where
\begin{equation*}
  \norm{\partial^{\beta_{1}}b(\xi)\,\partial^{\beta_{2}}b(\xi)} \lesssim \norm{\partial^{\beta_{1}}b(\xi)}\norm{\partial^{\beta_{2}}b(\xi)}\lesssim \left( 1 + \abs{\xi}^{2}\right)^{(r-(n+1))/2},
\end{equation*}
for $1 \le \abs{\beta_{1}}, \abs{\beta_{2}} \le n$, by the induction hypothesis. Thus, using again Lemma~\ref{lem:linear-algebra}, we conclude that
\begin{align*}
  \norm{\partial^{\alpha}b(\xi)} & \lesssim \norm{[b(\xi)]^{-1}} \norm{\partial^{\alpha}a(\xi) - \sum_{\abs{\beta_{1}}+ \abs{\beta_{2}} = n+1} \partial^{\beta_{1}}b(\xi)\,\partial^{\beta_{2}}b(\xi)} \\
    & \lesssim \left( 1 + \abs{\xi}^{2}\right)^{(r/2-(n+1))/2},
\end{align*}
which achieves the proof.
\end{proof}

We will now turn to the main result of this section which is the following theorem.

\begin{thm}\label{thm:smoothness-of-conjugates}
Let $A = a(D)$ be a Fourier multiplier of class $\mathcal{S}^{r}$ on $\Rd$, with $r \ge 1$. Let $q > 1 + d/2$ and $q \ge r$.
Then the mapping
\begin{equation*}
 \varphi \mapsto A_{\varphi} := R_{\varphi^{-1}}A R_{\varphi}, \qquad \DRd{q} \to \mathcal{L}(\HRd{q},\HRd{q-r}),
\end{equation*}
is smooth.
\end{thm}

The main ingredient of the proof is the following lemma that shows the boundedness of the $n$-th Fr\'{e}chet differentials of the operators $A_{{\varphi}}$.

\begin{lem}\label{lem:A_n}
Let $A=a(D)$ be a Fourier multiplier of class $\mathcal{S}^{r}$, $r \ge 1$ and let
\begin{equation*}
  A_{n}:= \partial^{n}_{\id}A_{\varphi} \in \mathcal{L}^{n+1}(\HRd{\infty},\HRd{\infty})
\end{equation*}
be the $(n+1)$-linear operator defined inductively by $A_{0} = A$ and
\begin{multline*}
    A_{n+1}(u_{0},u_{1}, \dotsc , u_{n+1}) = \nabla_{u_{n+1}} \left(A_{n}(u_{0}, u_{1}, \dotsc , u_{n}) \right)\\
    - \sum_{k=0}^{n} A_{n}(u_{0}, u_{1}, \dotsc ,\nabla_{u_{n+1}} u_{k}, \dotsc , u_{n}),
\end{multline*}
where $\nabla$ is the canonical derivative on $\Rd$. Then, each $A_n$ extends to a bounded multilinear operator
\begin{equation*}
  A_{n} \in \mathcal{L}^{n+1}(\HRd{q},\HRd{q-r}).
\end{equation*}
\end{lem}

\begin{rem}
For $n=1$, we have
\begin{equation*}
    A_{1}(u_{0},u_{1}) = [\nabla_{u_{1}},A]u_{0},
\end{equation*}
and for $n=2$, we get
\begin{equation*}
    A_{2}(u_{0},u_{1},u_{2}) = \big( [\nabla_{u_{2}},[\nabla_{u_{1}},A]] - [\nabla_{\nabla_{u_{2}}u_{1}},A] \big) u_{0}.
\end{equation*}
\end{rem}

The proof of Lemma~\ref{lem:A_n} relies on two main observations which proofs can be found in Appendix~\ref{sec:Fourier-multipliers-derivatives}. The \emph{first one} (Lemma~\ref{lem:multi-Fourier-multipliers}) is that the Fourier transform of
\begin{equation*}
  A_{n}(u_{0},u_{1}, \dotsc , u_{n}), \quad \text{where} \quad u_{0},u_{1}, \dotsc , u_{n} \in \HRd{\infty}
\end{equation*}
that we shall denote by $\widehat{A_{n}}$ (to avoid lengthy notation), can be written as
\begin{equation}\label{eq:integral-formula-for-An}
  \widehat{A_{n}}(\xi) = \int_{\xi_{0} + \dotsb + \xi_{n} = \xi} a_{n}(\xi_{0}, \dotsc ,\xi_{n}) \left[\hat{u}_{0}(\xi_{0}), \dotsc ,\hat{u}_{n}(\xi_{n}) \right]\, d\mu
\end{equation}
where $d\mu$ is the Lebesgue measure on the subspace $\xi_{0} + \dotsb + \xi_{n} = \xi$ of $(\RR^{d})^{n+1}$ and
\begin{equation*}
  a_{n} : (\RR^{d})^{n+1} \to \mathcal{L}^{n+1}(\CC^{d},\CC^{d})
\end{equation*}
is the $(n+1)$-linear map defined inductively by $a_{0} = a$ and
\begin{multline*}
  a_{n+1}(\xi_{0}, \dotsc ,\xi_{n+1}) = \\
  (2i\pi) \sum_{k=0}^{n} \Big[ a_{n}(\xi_{0}, \dotsc ,\xi_{k} + \xi_{n+1},\dotsc ,\xi_{n}) -  a_{n}(\xi_{0}, \dotsc ,\xi_{n}) \Big] \otimes \xi_{k}^{\sharp} ,
\end{multline*}
where $\xi^{\sharp}$ is the linear functional defined by $\xi^{\sharp}(X) := \xi \cdot X$. The \emph{second one} (Lemma~\ref{lem:nth-derivative-estimate}) is an estimate on the sequence $a_{n}$, namely
\begin{equation}\label{eq:an-estimate}
  \norm{a_{n}(\xi_{0}, \dotsc ,\xi_{n})} \le C_{n} \left(\prod_{k=0}^{n}\lambda_{1}(\xi_{k})\right)\left( \sum_{J \subset I_{n}} \lambda_{r-1}\big(\xi_{0} + \sum_{j \in J} \xi_{j}\big) \right),
\end{equation}
where $I_{n}:=\set{1, \dotsc ,n}$ and $\lambda_{r}(\xi) := (1 + \abs{\xi}^{2})^{r/2}$.

\begin{rem}
  An estimate for $A_{1}$ can be obtained directly using known estimates on \emph{commutators} (see~\cite{Lan2006} for an excellent exposition on the subject) but estimating $A_{n}$ requires to control a non-trivial sequence of iterated commutators. Hopefully, Lemma~\ref{lem:multi-Fourier-multipliers} and Lemma~\ref{lem:nth-derivative-estimate} allow us to avoid this painful task.
\end{rem}

Finally, the proof of Lemma~\ref{lem:A_n} also requires the following estimate, which can be found in~\cite[Lemma 2.3]{IKT2013}.

\begin{lem}\label{lem:pointwise_multiplication}
Let $q > d/2$ and $0 \le \rho \le q$. Then there exists $K >0$ such that for any $f\in H^{q}(\RR^{d},\RR)$, $g\in H^{\rho}(\RR^{d},\RR)$, the product $fg$ is in $H^{\rho}(\RR^{d},\RR)$ and
\begin{equation*}
  \norm{fg}_{H^{\rho}} \le K \norm{f}_{H^{q}} \norm{g}_{H^{\rho}}.
\end{equation*}
\end{lem}

\begin{proof}[Proof of Lemma~\ref{lem:A_n}]
Let $u_{0},u_{1}, \dotsc , u_{n} \in \HRd{\infty}$. We have
\begin{equation*}
  \norm{A_{n}(u_{0},u_{1}, \dotsc , u_{n})}^{2}_{H^{q-r}} = \int (\lambda_{q-r}(\xi))^{2} \abs{\widehat{A_{n}}(\xi)}^{2}\,d\xi.
\end{equation*}
But, due to~\eqref{eq:integral-formula-for-An} and~\eqref{eq:an-estimate}, we get
\begin{multline*}
  \abs{\widehat{A_{n}}(\xi)} \lesssim \\
  \sum_{J \subset I_{n}} \int_{\xi_{0} + \dotsb + \xi_{n} = \xi} \lambda_{r-1}\left(\xi_{0} + \sum_{j \in J} \xi_{j}\right) \abs{\widehat{\Lambda^{1}u_{0}}(\xi_{0})} \dotsb \abs{\widehat{\Lambda^{1}u_{n}}(\xi_{n})}\, d\mu,
\end{multline*}
because $\lambda_{r}(\xi)\abs{\hat{u}(\xi)} = \abs{\widehat{\Lambda^{r}u}(\xi)}$, where $\Lambda^{r} := \op{\lambda_{r}(\xi)}$. Now observe that given $n+1$ \emph{scalar-valued} functions $f_{0}, f_{1}, \dotsc ,f_{n} \in H^{\infty}(\RR^{d},\CC)$, we have
\begin{multline*}
  \mathfrak{F}\left(\Lambda^{r}(f_{0} f_{1} \dotsc f_{p}) f_{p+1} \dotsc f_{n}\right)(\xi) = \\
  \int_{\xi_{0} + \dotsb + \xi_{n} = \xi} \lambda_{r}\left(\xi_{0} + \xi_{1} + \dotsb \xi_{p}\right) \widehat{f_{0}}(\xi_{0}) \widehat{f_{1}}(\xi_{1}) \dotsb \widehat{f_{n}}(\xi_{n}) \, d\mu.
\end{multline*}
Thus, if we choose
\begin{equation*}
  f_{k} := \mathfrak{F}^{-1}\left(\abs{\widehat{\Lambda^{1}u_{k}}}\right), \qquad k=0,1, \dotsc ,n,
\end{equation*}
we get
\begin{equation*}
  \abs{\widehat{A_{n}}(\xi)} \lesssim \sum_{J \subset I_{n}} \abs{\mathfrak{F}\left(\Lambda^{r-1}\big(f_{0} \prod_{j\in J}f_{j}\big) \prod_{k\in J^{c}}f_{k}\right)(\xi)},
\end{equation*}
and therefore
\begin{equation*}
  \norm{A_{n}(u_{0},u_{1}, \dotsc , u_{n})}^{2}_{H^{q-r}} \lesssim \sum_{J \subset I_{n}} \norm{\Lambda^{r-1}\big(f_{0} \prod_{j\in J}f_{j}\big) \prod_{k\in J^{c}}f_{k}}^{2}_{H^{q-r}}.
\end{equation*}
Finally, by virtue of lemma~\ref{lem:pointwise_multiplication}, and because we assume $q > 1+d/2$ and $r \ge 1$, we have
\begin{align*}
  \norm{\Lambda^{r-1}(f_{0} \prod_{j\in J}f_{j}) \prod_{k\in J^{c}}f_{k}}^{2}_{H^{q-r}}
  & \lesssim \norm{\Lambda^{r-1}\big(f_{0} \prod_{j\in J}f_{j}\big)}^{2}_{H^{q-r}} \norm{\prod_{k\in J^{c}}f_{k}}^{2}_{H^{q-1}} \\
  & \lesssim \norm{f_{0} \prod_{j\in J}f_{j}}^{2}_{H^{q-1}} \norm{\prod_{k\in J^{c}}f_{k}}^{2}_{H^{q-1}} \\
  & \lesssim \prod_{k=0}^{n} \norm{f_{k}}^{2}_{H^{q-1}} = \prod_{k=0}^{n} \norm{u_{k}}^{2}_{H^{q}},
\end{align*}
which completes the proof.
\end{proof}

\begin{proof}[Proof of Theorem~\ref{thm:smoothness-of-conjugates}]
Using Lemma~\ref{lem:A_n} and~\ref{lem:composition}, the proof of Theorem~\ref{thm:smoothness-of-conjugates} follows similarly as in~\cite{EK2014}. The basic idea is to show that the mapping
\begin{equation*}
 \varphi \mapsto A_{\varphi} := R_{\varphi^{-1}}A R_{\varphi}, \qquad \DRd{q} \to \mathcal{L}(\HRd{q},\HRd{q-r}),
\end{equation*}
is smooth, if and only if, each $A_n$ extends to a bounded $(n+1)$-linear operator in $\mathcal{L}^{n+1}(\HRd{q},\HRd{q-r})$.
Therefore we need the local boundedness of the composition operator $R_{\varphi}$, c.f.~\cite[Theorem 3.4.]{EK2014}.
The proof of this statement does not depend on the dimension of the base manifold and thus we will not repeat the argumentation.
Now the statement of the Theorem follows using Lemma~\ref{lem:A_n}.
\end{proof}

\section{Smoothness of the extended metric on Hilbert manifolds}
\label{sec:smoothness}

For general facts on Riemannian geometry on a Banach manifold we refer to~\cite{Lan1999}. Let us recall that a Riemannian metric $G$ on $\DRd{q}$, where $q > 1 + d/2$,
is a smooth, symmetric, positive definite, covariant $2$-tensor field on $\DRd{q}$. In other words, we have for each $\varphi \in \DRd{q}$ a symmetric, positive definite,
bounded, bilinear form $G_{\varphi}$ on $T_{\varphi}\DRd{q}$ and, in any local chart $U$, the mapping
\begin{equation*}
  \varphi \to G_{\varphi}, \qquad U \to \mathcal{L}_{\mathrm{Sym}}^{2}(\HRd{q},\RR)
\end{equation*}
is smooth. Given any $\varphi\in \DRd{q}$, we can therefore consider the bounded, linear operator
\begin{equation*}
  \tilde{G}_{\varphi} : T_{\varphi}\DRd{q} \to T_{\varphi}^{*}\DRd{q},
\end{equation*}
called the \emph{flat map} and defined by $\tilde{G}_{\varphi}(v) := G_{\varphi}(v, \cdot)$. The metric is \emph{strong} if $\tilde{G}_{\varphi}$ is a topological linear isomorphism for every $\varphi \in \DRd{q}$, whereas it is \emph{weak} if $\tilde{G}_{\varphi}$ is only injective for some $\varphi \in \DRd{q}$.

Suppose that $A: \HRd{\infty} \to \HRd{\infty}$ is a $L^{2}$ symmetric, positive definite, topological isomorphism, that extends to a bounded, \emph{injective} operator in the space
\begin{equation*}
  \mathcal{L}(\HRd{q},\HRd{-q}),
\end{equation*}
where $\HRd{-q}$ is the Banach dual of $\HRd{q}$. Then $A$ induces an inner product on the Sobolev space $\HRd{q}$, given by
\begin{equation*}
  \langle u_{1},u_{2}\rangle := \left( Au_{1}, u_{2}\right),
\end{equation*}
where $\left( \cdot , \cdot \right)$ denotes the dual pairing between $\HRd{q}$ and its topological dual $\HRd{-q}$. To conclude that the family of inner products
\begin{equation*}
  G_{\ph}(v_{1},v_{2})= \langle v_{1}\circ\ph^{-1}, v_{2}\circ\ph^{-1}\rangle,
\end{equation*}
where $\ph\in \DRd{q}$ and $v_{1},v_{2}\in T_{\ph}\DRd{q}$, defines a (smooth) Riemannian metric on $\DRd{q}$ we need to show that the corresponding flat map
\begin{equation*}
  \ph \mapsto \tilde G_{\ph}, \qquad \DRd{q} \to \mathcal{L}(\HRd{q},\HRd{-q})
\end{equation*}
is a smooth mapping. We will distinguish two cases:
\begin{enumerate}
 \item \emph{Weak metrics on $\DRd{q}$}; i.e., the metric $G$ is only injective, seen as a mapping from the tangent bundle $T\DRd{q}$ to the co-tangent bundle $T^*\DRd{q}$
 \item \emph{Strong metrics on $\DRd{q}$}; i.e., the metric $G$ induces an isomorphism between the tangent bundle and the co-tangent bundle.
\end{enumerate}

\subsection{Weak metrics on $\DRd{q}$}

We consider first the case of an inner product on $\HRd{\infty}$ which is given by
\begin{equation*}
  \langle u_{1},u_{2}\rangle = \int Au_{1} \cdot u_{2} \, dx ,
\end{equation*}
where $A$ is a Fourier multiplier of class $\mathcal{S}^{r}$ and $r \ge 1$. Let $q > 1 +d/2$ and $q-r \ge 0$. We suppose further that the inertia operator $A$ extends
to a bounded \emph{isomorphism} between $\HRd{q}$ and $\HRd{q-r}$. Since $\HRd{q-r}$ embeds continuously into $\HRd{-q}$,
it is sufficient to show that the mapping
\begin{equation}\label{equation:inertiatimesJacobian}
\ph \mapsto \tilde G_{\ph}= J_{\ph} A_\ph, \qquad \DRd{q} \rightarrow \mathcal L(\HRd{q},\HRd{q-r})
\end{equation}
is a smooth mapping to conclude that $A$ induces a smooth (weak) Riemannian metric on $\DRd{q}$.
Here $J_{\ph}$ denotes the Jacobian determinant of $\ph$.

\begin{lem}
Let $q > 1 + d/2$ and $q-r \ge 0$. Then the mapping
\begin{equation*}
  \varphi \mapsto A_{\varphi}, \quad \DRd{q} \to \mathcal{L}(\HRd{q},\HRd{q-r})
\end{equation*}
is smooth if and only if the mapping
\begin{equation*}
  \varphi \mapsto \tilde G_{\varphi}, \quad \DRd{q} \to \mathcal{L}(\HRd{q},\HRd{q-r})
\end{equation*}
is smooth and $A$ induces a smooth Riemannian metric on $\DRd{q}$.
\end{lem}

\begin{proof}
This lemma is an immediate consequence of the fact that pointwise multiplication
\begin{equation*}
  (f,u) \mapsto fu, \qquad H^{q-1}(\Rd, \RR) \times \HRd{q-r} \to \HRd{q-r}
\end{equation*}
is continuous for $q > 1 + d/2$ and $q-r \ge 0$.
\end{proof}

\begin{rem}
If the mapping
\begin{equation*}
  \varphi \mapsto \tilde A_{\varphi}, \quad \DRd{q} \to \mathcal{L}(\HRd{q},\HRd{p})
\end{equation*}
is a smooth mapping for some $p\leq q$, then it is also a smooth mapping
\begin{equation*}
\DRd{q} \to \mathcal{L}(\HRd{q},\HRd{\tilde p})
\end{equation*}
for each $\tilde p\leq p$.
\end{rem}
\begin{rem}
If $G$ is a weak Riemannian metric on $\DRd{q}$ it is also a weak Riemannian metric on $\DRd{\tilde q}$, for $\tilde q \geq q$. The converse of this statement does not hold in general.
\end{rem}

As a consequence of Lemma~\ref{thm:smoothness-of-conjugates} we obtain the following result concerning the smoothness of the metric $G$.

\begin{thm}[Smoothness of the metric]\label{thm:smoothness-weak-metric}
Let $A$ be a Fourier multiplier of class $\mathcal{S}^{r}$ and $r \ge 1$. Let $q>1+\frac{d}{2}$ and $q-r \ge 0$. Then, the right-invariant, weak Riemannian metric defined on $\DiffRd$ extends to a smooth, \emph{weak} Riemannian metric on the Banach manifold $\DRd{q}$.
\end{thm}

\begin{proof}
The proof of this theorem is an immediate consequence of theorem~\ref{thm:smoothness-of-conjugates}.
\end{proof}

\subsection{Strong metrics on $\DRd{s}$}\label{sec:strongmetric}

We now consider the case where the order of $A$ is high enough such that it induces a strong Riemannian metric on the
Sobolev completion $\DRd{s}$. Therefore let $A$ be a Fourier multiplier of class $\mathcal{E}^{2s}$ with $s > 1+\frac{d}{2}$ that
extends to a bounded \emph{isomorphism} between $\HRd{s}$ and its \emph{dual space} $\HRd{-s}$. If the mapping
\begin{equation*}
\ph \mapsto \tilde G_{\ph}= J_{\ph} A_\ph, \qquad \DRd{s} \rightarrow \mathcal L(\HRd{s},\HRd{-s})
\end{equation*}
is a smooth mapping, then $A$ induces a strong Riemannian metric on $\DRd{s}$.

Using Lemma~\ref{lem:square-root}, we can decompose the operator $A$ as
\begin{equation*}
  A := B^{*}B : \HRd{s} \to \HRd{-s}
\end{equation*}
where $B$ is a Fourier multiplier of class $\mathcal{E}^{s}$ and $B^{*}$ is the corresponding transpose of the operator $B$:
\begin{equation*}
  B : \HRd{s} \to  \LRd,\qquad B^{*} : \LRd \to \HRd{-s}.
\end{equation*}
We can therefore rewrite the metric on $\DRd{q}$ as
\begin{align*}
  G_{\ph}( v_{1},v_{2}) & = \int_{\Rd} B (v_{1}\circ \varphi^{-1}) \cdot B (v_{2}\circ \varphi^{-1}) \, dx
  \\
  & = \int_{\Rd} B_\ph  (v_{1}) \cdot B_\ph (v_{2})J_{\ph} \, dx
\end{align*}
for $v_{1}$, $v_{2}$ tangent vectors in $T_{\varphi}\DRd{s}$. Here $B_\ph$ denotes the operator
\begin{equation*}
  B_{\varphi} := R_{\varphi} \circ B \circ R_{\varphi^{-1}}.
\end{equation*}
Using the transpose of the operator $B_\ph$ we obtain
\begin{align*}
  G_{\ph}( v_{1},v_{2}) = \left(B_\ph^* \circ M_{J_{\varphi}}\circ B_\ph  (v_{1}), v_{2}\right)_{\HRd{q}\times \HRd{-q}},
\end{align*}
where $M_{J_{\varphi}}$ is the pointwise multiplication by the Jacobian determinant $J_{\varphi}$ of the diffeomorphism $\varphi$.
Comparing this with \eqref{equation:inertiatimesJacobian}, we obtain
\begin{align*}
  B_\ph^* \circ M_{J_{\varphi}}\circ B_\ph = J_{\ph}A_{\ph} = \tilde G_{\ph}\,.
\end{align*}

The latter formula can now be used to obtain the following result concerning the smoothness of this family of inner products.

\begin{thm}[Smoothness of the strong metric]\label{thm:smoothness-strong-metric}
Let $s > 1 + d/2$ and $B \in \operatorname{Isom}(\HRd{s},\LRd)$. Suppose that the mapping
\begin{equation*}
  \varphi \mapsto B_{\varphi}, \quad \DRd{s} \to \mathcal{L}(\HRd{s},\LRd)
\end{equation*}
is smooth. Then the inertia operator
\begin{equation*}
  A := B^{*}B : \HRd{s} \to \HRd{-s}
\end{equation*}
induces a \emph{smooth and strong} Riemannian metric on $\DRd{s}$.
\end{thm}

\begin{proof}
Note first that, given $\varphi \in \DRd{s}$, the mapping $B_{\varphi}^{*}\circ M_{J_{\varphi}}\circ B_{\varphi}$ is a topological isomorphism between $\HRd{s}$ and $\HRd{-s}$. The smoothness is proved as follows. Since transposition and composition of bounded operators between Banach spaces are themselves bounded operators it follows that
the transpose
\begin{equation*}
  \varphi \mapsto B^{*}_{\varphi}, \quad \DRd{s} \to \mathcal{L}(\LRd,\HRd{-s})
\end{equation*}
is smooth iff
\begin{equation*}
  \varphi \mapsto B_{\varphi}, \quad \DRd{s} \to \mathcal{L}(\HRd{s},\LRd)
\end{equation*}
is smooth. Using that
\begin{equation*}
  \varphi \mapsto M_{J_{\varphi}} , \quad \DRd{s} \to \mathcal{L}(\LRd,\LRd),
\end{equation*}
is smooth for $s >1 + d/2$, it follows that the composition
\begin{equation*}
  B_{\varphi}^{*}\circ M_{J_{\varphi}}\circ B_{\varphi} , \quad \DRd{s} \to \mathcal{L}(\HRd{s},\HRd{-s}),
\end{equation*}
is smooth and that thus also the metric is smooth. The second statement follows directly.
\end{proof}

\section{Local well-posedness of the geodesic equation}
\label{sec:well-posedness}

In this section we show the local well-posedness of the EPDiff equation, assuming a sufficiently high order $r$ of the metric $G$. It turns out that the required order $r$ does not depend on the dimension $d$. The proof is based on the same method as in the seminal article of Ebin and Marsden~\cite{EM1970}. Therefore we first need to show the smoothness of the extended spray. Then the local well-posedness follows from the Picard-Lindel\"{o}f (or Cauchy-Lipschitz) theorem.

\subsection{Smoothness of the extended spray}

We will now prove smoothness of the spray on $T\DRd{q}$, when the inertia operator $A$ is in the class $\mathcal{E}^{r}$ (see definition \ref{def:class-Er}) where $r \ge 1$ and $q > 1 +
d/2$, with $q-r \ge 0$.

\begin{thm}\label{thm:smoothness_spray}
Let $A$ be a Fourier multiplier in the class $\mathcal{E}^{r}$, where $r \ge 1$ and let $q > 1 +
d/2$, with $q-r \ge 0$. Then the geodesic spray
\begin{equation*}
    (\varphi, v) \mapsto S_{\varphi} (v) = R_{\varphi} \circ S \circ
R_{\varphi^{-1}} (v),
\end{equation*}
where
\begin{equation}\label{eq:spray}
    S(u) = A^{-1} \left\{ [A,\nabla_{u}] u - (\nabla u )^{t} Au - (\dive u) Au \right\}
\end{equation}
extends smoothly to $T\DRd{q} = \DRd{q}\times \HRd{q}$.
\end{thm}

\begin{rem}
Note that this statement is highly non-trivial, since the metric is only a \emph{weak metric}. In Lemma \ref{lem:local-well_strong}, where we treat the strong metric case, we will obtain the smoothness of the spray for free. Furthermore the assumption on the order of the operator is sharp, i.e., for operators $A$ of order $r < 1$ the geodesic spray can never extend smoothly to some Sobolev completion $T\DRd{q}$. This follows immediately from equation \eqref{eq:spray}: the term $\dive u$, which is of order one, is always present in this equation. Thus we need the operator $A^{-1}$ to be a smoothing operator of at least order one.
\end{rem}

\begin{proof}[Proof of Theorem~\ref{thm:smoothness_spray}]
Set
\begin{equation*}
  Q^{1}(u) := [A,\nabla_{u}] u, \quad Q^{2}(u) := (\nabla u )^{t} Au,
\quad Q^{3}(u) := (\dive u) Au .
\end{equation*}
Then
\begin{equation*}
    S_{\varphi}(v) = A_{\varphi}^{-1} \left\{ Q^{1}_{\varphi}(v) -
Q^{2}_{\varphi}(v) - Q^{3}_{\varphi}(v) \right\},
\end{equation*}
and the proof reduces to establish, using the chain rule, that the mappings
\begin{equation*}
  (\varphi,v) \mapsto Q^{i}_{\varphi}(v), \quad \text{and} \quad
(\varphi,w) \mapsto A_{\varphi}^{-1} (w)
\end{equation*}
are smooth, for $i=1,2,3$.

(a) By virtue of Lemma \ref{lem:A_n}, we have
\begin{equation*}
    \partial_{\varphi}A_{\varphi} (v,v) = A_{1,\varphi}(v,v) = -
Q^{1}_{\varphi}(v),
\end{equation*}
and therefore
\begin{equation*}
    (\varphi,v) \mapsto  Q^{1}_{\varphi}(v), \quad \DRd{q}\times\HRd{q} \to \HRd{q-r}
\end{equation*}
is smooth.

(b) We have $Q^{2}_{\varphi}(v) = \big( \nabla (v\circ\varphi^{-1})
\big)^{t}\circ \varphi . A_{\varphi}(v)$. But, in a local chart, we get
\begin{equation*}
  \left[\big( \nabla (v\circ\varphi^{-1}) \big)^{t}\circ
\varphi\right]_{k}^{i}(x) = \delta^{ij}\delta_{kl}
\left[(d\varphi(x))^{-1}\right]_{j}^{m} \partial_{m}v^{l}(x),
\end{equation*}
where $(d\varphi(x))^{-1}$ is the (pointwise) inverse of the invertible matrix $d\varphi(x)$. Its coefficients are therefore polynomial expressions of the partial derivatives $\partial_{p}\varphi^{q}$ divided by the Jacobian $J_{\varphi}$. Thus
\begin{gather*}
  (\varphi,v) \mapsto  \big( \nabla (v\circ\varphi^{-1}) \big)^{t}\circ \varphi \\
  \DRd{q}\times\HRd{q} \to \mathcal{L}(\HRd{q-1},\HRd{q-1})
\end{gather*}
is smooth since $\HRd{q-1}$ is a multiplicative algebra for $q > 1+d/2$.
To conclude that
\begin{equation*}
    (\varphi,v) \mapsto  Q^{2}_{\varphi}(v), \quad \DRd{q}\times\HRd{q}
\to \HRd{q-r}
\end{equation*}
is smooth, we use the fact that pointwise multiplication extends to a
bounded bilinear mapping
\begin{equation*}
  \HRd{q-1} \times \HRd{q-r} \to \HRd{q-r},
\end{equation*}
if $q-1 > d/2$ and $0 \le q-r \le q-1$ (c.f. Lemma \ref{lem:sobolev_embedding}) and that
\begin{equation*}
    (\varphi,v) \mapsto  A_{\varphi}(v), \quad \DRd{q}\times\HRd{q} \to
\HRd{q-r}
\end{equation*}
is smooth by hypothesis.

(c) We have $Q^{3}_{\varphi}(v) = \big(
\dive(v\circ\varphi^{-1})\circ\varphi \big)  A_{\varphi} (v)$. But
\begin{equation*}
    \dive(v\circ\varphi^{-1})\circ\varphi =
\left[(d\varphi)^{-1}\right]_{i}^{j} \partial_{j}v^{i}
\end{equation*}
and we conclude as in (b) that
\begin{equation*}
    (\varphi,v) \mapsto \dive(v\circ\varphi^{-1})\circ\varphi, \quad \DRd{q}\times\HRd{q} \to
H^{q-1}(\Rd,\RR)
\end{equation*}
is smooth and that
\begin{equation*}
    (\varphi,v) \mapsto  Q^{3}_{\varphi}(v), \quad \DRd{q}\times\HRd{q}
\to \HRd{q-r}
\end{equation*}
is smooth.

(d) The set
\begin{equation*}
  \mathrm{Isom}(\HRd{q},\HRd{q-r})
\end{equation*}
is open in
\begin{equation*}
  \mathcal{L}(\HRd{q},\HRd{q-r})
\end{equation*}
and the mapping
\begin{gather*}
  P \mapsto P^{-1}, \\
  \mathrm{Isom}(\HRd{q},\HRd{q-r}) \to \mathcal{L}(\HRd{q-r},\HRd{q})
\end{gather*}
is smooth (even real analytic). Besides
\begin{equation*}
  A_{\varphi} \in \mathrm{Isom}(\HRd{q},\HRd{q-r}),
\end{equation*}
for all $\varphi \in \DRd{q}$, and the mapping
\begin{equation*}
    \varphi \mapsto A_{\varphi}, \quad \DRd{q} \to
\mathrm{Isom}(\HRd{q},\HRd{q-r})
\end{equation*}
is smooth. Thus
\begin{equation*}
    (\varphi,w) \mapsto  A^{-1}_{\varphi}(w), \quad  \HRd{q-r} \to \HRd{q}
\end{equation*}
is smooth.
\end{proof}

\subsection{Local well-posedness}

The local existence of geodesics on the Hilbert manifold $\DRd{q}$
follows from the Picard-Lindel\"{o}f (or Cauchy-Lipschitz) theorem, due to
the smoothness of the extended spray on $\DRd{q}$:

\begin{thm}\label{thm:smooth_flow-Hq}
Let $A$ be a Fourier multiplier in the class $\mathcal{E}^{r}$ with $r \ge 1$. Let $q > 1 + d/2$ with $q-r \ge 0$. Consider the geodesic flow on the tangent bundle $T\DRd{q}$ induced by the inertia operator $A$. Then, given any $(\varphi_{0},v_{0})\in
T\DRd{q}$, there exists a unique non-extendable geodesic
\begin{equation*}
  (\varphi, v)\in C^\infty(J,T\DRd{q})
\end{equation*}
on the maximal interval of existence $J$, which is open and contains $0$.
\end{thm}

And we obtain well-posedness of the Euler equation in $\HRd{q}$.

\begin{cor}\label{cor:Euler-well-posedness-Hq}
The corresponding Euler equation has, for any initial data
$u_{0}\in\HRd{q}$, a unique non-extendable smooth solution
\begin{equation*}
  u\in C^{0}(J,\HRd{q}) \cap C^{1}(J,\HRd{q-1}).
\end{equation*}
The maximal interval of existence $J$ is open and contains $0$.
\end{cor}

The remarkable observation that the maximal interval of existence is independent of the
parameter $q$, due to the right-invariance of the spray (cf.
lemma~\ref{lem:noloss}) was pointed out in~\cite[Theorem 12.1]{EM1970}.
This makes it possible to avoid Nash--Moser type schemes to prove local
existence of smooth geodesics in the smooth category.

\begin{lem}[No loss, nor gain]\label{lem:noloss}
Given $(\varphi_{0},v_{0})\in T\DRd{q+1}$, we have
\begin{equation*}
  J_{q+1}(\varphi_{0},v_{0})= J_{q}(\varphi_{0},v_{0}),
\end{equation*}
for $q> 1 + d/2$ and $q-r \ge 0$.
\end{lem}

\begin{proof}
Let $u$ be a \emph{constant} vector field on $\RR^{d}$. Its flow
\begin{equation*}
  \psi^{u}_{s}(x) := x + su
\end{equation*}
does not belongs to $\DRd{q}$ because $u \notin \HRd{q}$. However, if $\varphi\in\DRd{q}$, then $\varphi\circ\psi^{u}_{s} \in \DRd{q}$ and
the flow $\psi^{u}_{s}$ may be considered as a one parameter group of (smooth) isometries of the (weak) Riemannian manifold $\DRd{q}$. The action on $\DRd{q}$ is defined by
\begin{equation*}
  \big(\psi^{u}_{s} \cdot \varphi \big)(x): = \varphi(x+su),\quad
\varphi \in \DRd{q},\quad x\in\RR.
\end{equation*}
and the induced action on $T\DRd{q}$ is given by
\begin{equation*}
  \big(\psi^{u}_{s} \cdot (\varphi,v)\big)(x): = (\varphi(x+su),v(x+su)),\quad (\varphi,v)\in T\DRd{q},\quad x\in\RR.
\end{equation*}
Because $\psi^{u}_{s}$ is a Riemannian isometry, the geodesic spray $F_{q}$ is invariant under $\psi^{u}_{s}$ and the same is true for its flow $\Phi_{q}$. Hence
\begin{equation*}
  \Phi_{q}(t,\psi^{u}_{s} \cdot (\varphi_{0},v_{0})) = \psi^{u}_{s} \cdot \Phi_{q}(t,(\varphi_{0},v_{0})),
\end{equation*}
for all $t\in J_{q}(\varphi_{0},v_{0})$ and $s\in\RR$.

Now, note that if $(\varphi,v)\in T\DRd{q+1}$, then\footnote{We will avoid to
write $T\psi^{u}_{s}$, $T(TR\psi^{u}_{s})$, \dots and simply keep the
notation $\psi^{u}_{s}$.}
\begin{equation*}
  s \mapsto \psi^{u}_{s} \cdot (\varphi,v), \quad \RR \to T\DRd{q}
\end{equation*}
is a $C^{1}$ mapping, and that
\begin{equation*}
  \frac{d}{ds} \psi^{u}_{s}\cdot (\varphi,v) = (T\varphi . u, \nabla_{u} v).
\end{equation*}
Therefore, if $(\varphi_{0},v_{0})\in T\DRd{q+1}$, we get
\begin{equation*}
    \left.\frac{d}{ds}\right\vert_{s=0} \Phi_{q}(t, \psi^{u}_{s} \cdot
(\varphi_{0},v_{0})) =
\partial_{(\varphi,v)}\Phi_{q}(t,(\varphi_{0},v_{0})).(T\varphi_{0} . u,
\nabla_{u} v_{0}),
\end{equation*}
and thus
\begin{equation*}
  \partial_{(\varphi,v)}\Phi_{q}(t,(\varphi_{0},v_{0})).(T\varphi_{0} . u, \nabla_{u} v_{0}) = (T\varphi(t) . u, \nabla_{u} v(t)).
\end{equation*}
But
\begin{equation*}
  \partial_{(\varphi,v)}\Phi_{q}(t,(\varphi_{0},v_{0})).(T\varphi_{0} . u, \nabla_{u} v_{0}) \in \HRd{q}\times\HRd{q},
\end{equation*}
for all constant vector field $u$, and hence
\begin{equation*}
  (\varphi(t),v(t))\in T\DRd{q+1}\quad\text{for all}\quad t\in J_{q}(\varphi_{0},v_{0}).
\end{equation*}
We conclude therefore that
\begin{equation*}
    J_{q}(\varphi_{0},v_{0}) = J_{q+1}(\varphi_{0},v_{0}),
\end{equation*}
which completes the proof.
\end{proof}

\begin{rem}\label{rem:iterated_noloss}
An iteration of the above argument shows that for initial data $(\varphi_{0},v_{0})\in T\DRd{q+k}$ with $k\geq1$, we also have
\begin{equation*}
  J_{q+k}(\varphi_{0},v_{0})= J_{q}(\varphi_{0},v_{0}).
\end{equation*}
\end{rem}

\begin{rem}\label{rem:noloss_nogain}
Lemma~\ref{lem:noloss} states that there is no loss of spatial
regularity during the evolution. By reversing the time direction, it
follows from the unique solvability that there is also no gain of
regularity in the following sense: Let $(\varphi_{0},v_{0})\in T\DRd{q}$
be given and assume that $(\varphi(t_{1}),v(t_{1}))\in T\DRd{q+1}$ for some
$t_{1}\in J_{q}(\varphi_{0},v_{0})$. Then $(\varphi_{0},v_{0})\in
T\DRd{q+1}$.
\end{rem}

We get therefore the following local existence result.

\begin{thm}\label{thm:smooth_flow-Hinfty}
Let $A$ be a Fourier multiplier in the class $\mathcal{E}^{r}$ with $r \ge 1$ and
consider the geodesic flow on the tangent bundle $T\DiffRd$. Then, given any $(\varphi_{0},v_{0})\in
T\DiffRd$, there exists a unique non-extendable geodesic
\begin{equation*}
  (\varphi, v)\in C^\infty(J,T\DiffRd)
\end{equation*}
on the maximal interval of existence $J$, which is open and contains $0$.
\end{thm}
\begin{rem}
Note, that this does not prove positivity of the induced geodesic distance. In fact, it has been shown in~\cite{BBHM2013,BBM2013,BBHM2012}, that there are inertia operators with symbols in $\mathcal{S}^{1}$ such that the induced geodesic distance vanishes identically on $\operatorname{Diff}(S^1)$.
For diffeomorphism groups on general manifolds, it has only been shown that the geodesic distance vanishes if the inertia operator is of class
$\mathcal{S}^{r}$, $r>1$. The boundary case $r=1$ remains open so far.
\end{rem}

We also obtain well-posedness of the Euler equation.

\begin{cor}\label{cor:Euler-well-posedness-Hinfty}
The corresponding Euler equation has for any initial data
$u_{0}\in\CS(\RR^{d})$ a unique non-extendable smooth solution
\begin{equation*}
  u\in C^\infty(J,\CS(\RR^{d})).
\end{equation*}
The maximal interval of existence $J$ is open and contains $0$.
\end{cor}

\section{Global well-posedness}
\label{sec:global-well-posedness}

In this section we let $s > 1 + d/2$ and $A \in \operatorname{Isom}(\HRd{s},\HRd{-s})$ be an invertible Fourier multiplier of class $\mathcal{E}^{2s}$. Similarly, as in Section~\ref{sec:strongmetric}, we can decompose the operator as $A := B^{*}B$. According to Theorem~\ref{thm:smoothness-strong-metric}, the operator $A := B^{*}B$ induces a \emph{strong and smooth} Riemannian metric on the Hilbert manifold $\DRd{s}$. In that case, the associated spray is smooth (see~\cite{Lan1999} for instance). Thus we obtain the local well-posedness on the Hilbert manifold $\DRd{s}$:

\begin{lem}\label{lem:local-well_strong}
Let $A=B^{*}B $ be an invertible Fourier multiplier of class $\mathcal{E}^{2s}$ with $s > 1 + d/2$. Consider the geodesic flow on the tangent bundle $T\DRd{s}$ induced by the inertia operator $A$. Then, given any $(\varphi_{0},v_{0})\in T\DRd{s}$, there exists a unique non-extendable geodesic
\begin{equation*}
  (\varphi, v)\in C^\infty(J,T\DRd{s})
\end{equation*}
on the maximal interval of existence $J$, which is open and contains $0$.
\end{lem}

For the strong metric case we obtain a much stronger result, namely global existence of geodesics:

\begin{thm}
Let $A=B^{*}B $ be an invertible Fourier multiplier of class $\mathcal{E}^{2s}$ with $s > 1 + d/2$, such that $A\in \operatorname{Isom}(\HRd{s},\HRd{-s})$. Then, given any $(\varphi_{0},v_{0})\in T\DRd{s}$, there exists a unique geodesic
\begin{equation*}
  (\varphi, v)\in C^\infty(\mathbb R,T\DRd{s})
\end{equation*}
that is defined for all time $t\in \mathbb R$.
\end{thm}

\begin{proof}
In~\cite[Lemma 5.2]{GMR2009} it has been shown that any manifold that is in addition a topological group with smooth right-translation, equipped with a smooth strong metric is geodesically complete. We have shown the smoothness of the metric in Theorem~\ref{thm:smoothness-strong-metric} and thus the result follows.
\end{proof}

\begin{rem}
The smoothness of integral order metrics -- and thus the global well-posedness of the geodesic equation for these metrics -- has been already observed in the article~\cite{EM1970}.
In~\cite[Corollary 7.5]{BV2014} it has been shown that for any smooth and strong Riemannian metric $G$ on $\DRd{s}$ all statements of the theorem of Hopf--Rinow hold, i.e.,
\begin{enumerate}
\item The space $(\DRd{s},G)$ is geodesically complete.
\item The space $(\DRd{s}_0,\operatorname{dist}^G)$ is metrically complete.
\item Any two diffeomorphisms in $\DRd{s}_0$ can be connected by a minimizing geodesic.
\end{enumerate}
Here $\DRd{s}_0\subset \DRd{s}$ denotes the connected component of the identity. However, the smoothness of fractional order metrics has been left open in these articles.
\end{rem}

Using the No--loss--no--gain--Lemma (c.f. Lemma \ref{lem:noloss} and Remark \ref{rem:iterated_noloss}), we are able to
transport the result to the smooth category:

\begin{thm}\label{thm:geod_complete-Hinfty}
Let $A=B^{*}B $ be an invertible Fourier multiplier of class $\mathcal{E}^{2s}$ with $s > 1 + d/2$, such that
$A\in \operatorname{Isom}(\HRd{s},\HRd{-s})$. Then, given any $(\varphi_{0},v_{0})\in
T\DiffRd$, there exists a unique geodesic
\begin{equation*}
  (\varphi, v)\in C^\infty(\mathbb R,T\DiffRd)
\end{equation*}
that is defined for all time $t\in \mathbb R$.
\end{thm}

\begin{rem}
Note, that in the smooth category we do not obtain the result that any two diffeomorphisms can be connected by a minimizing geodesic.
\end{rem}

We finally obtain the global well-posedness of the Euler equation.

\begin{cor}\label{cor:Euler-global-well-posedness-Hinfty}
Given the assumptions of Theorem~\ref{thm:geod_complete-Hinfty}, the corresponding Euler equation has for any initial data
$u_{0}\in\CS(\RR^{d})$ a unique smooth solution
\begin{equation*}
  u\in C^\infty(\mathbb R,\CS(\RR^{d})).
\end{equation*}
that is defined for all time $t$.
\end{cor}

\begin{rem}
Note, that the results of this Section apply in particular to the $H^{s}$-metric for $s > 1 + d/2$.
\end{rem}

\begin{rem}
In contrast to the weak metric case, the non-vanishing of the geodesic distance is guaranteed for any strong Riemannian metric, c.f.~\cite{Lan1999}.
\end{rem}

\section{Conclusions and Outlook}
\label{sec:conclusion}

In this article we study right-invariant metrics induced by Fourier multipliers on the diffeomorphism group of $\Rd$. This class of metrics includes in particular the family of fractional order Sobolev type metrics. We prove that, under certain conditions on the Fourier multiplier, the metric extends to a smooth metric on the Sobolev completion $\DRd{q}$ (for sufficiently high $q$). Subsequently we use this result to prove local well-posedness of the corresponding Euler equations on $\DRd{q}$, using a method developed by Ebin and Marsden in~\cite{EM1970}. Observing that there is neither loss nor gain in regularity during the geodesic evolution we are able to transfer this result to the smooth category, i.e., we obtain a local well posedness result also on the Lie group $\DiffRd$.

For metrics of order $s>\frac d2+1$ we prove that they even induce a strong and smooth Riemannian metric on $\DRd{s}$. Combining the right--invariance of the metric with this result we obtain global well-posedness of the geodesic equation --
both in the smooth category and on $\DRd{s}$.

Although all of our results and proofs are formulated for the diffeomorphism group of $\Rd$ they directly translate to the diffeomorphism group of the $d$-dimensional torus. In the case
$d=1$ -- i.e., $\operatorname{Diff}(S^1)$ -- our results yield a combination of the results of~\cite{EK2014,EK2014a}.

In future work it would be interesting to generalize these results to fractional order metrics on diffeomorphism groups of general manifolds. The main obstacle towards such a result is to show that the metric extends smoothly to some Sobolev completion $\mathcal D^q(M)$. This result can be proven, for metrics that are induced by a differential operator. For more general metrics one would need to derive iterated commutator estimates for elliptic Pseudo-differential operators.

Another interesting research topic is the blow-up behaviour of the boundary case. It is well-known that the Camassa-Holm equation allows solutions to develop singularities in finite time~\cite{Con2001}.
This equation corresponds to the $H^1$-metric on $\operatorname{Diff}(S^1)$. It is proven that the geodesic equation is globally well-posed for metrics of order $s>\frac32$. The question of the occurrence of blow-up
along geodesics for metrics of order $s=\frac32$ remains open.

\appendix

\section{Translation invariant operators}
\label{sec:translation-invariant-operators}

Let $\mathcal{S}(\Rd,\Rd)$ denote the Fr\'{e}chet space of all rapidly decreasing smooth vector fields on $\Rd$. We define $\mathcal{O}_{M}$ as the space of all $\mathcal{L}(\CC^{d})$-valued slowly growing functions on $\Rd$, i.e. a smooth function $a\in C^\infty(\Rd,\mathcal{L}(\CC^{d}))$ belongs to $\mathcal{O}_{M}$ iff given any $\alpha\in\mathbb{N}^{d}$ there is $m_\alpha\in \mathbb{N}$ and $C_\alpha>0$ such that
\begin{equation*}
\Vert \partial ^\alpha a(\xi)\Vert_{\mathcal{L}(\CC^{d})}\le C_\alpha(1+\abs{\xi}^{2})^{m_{\alpha}/2},\quad \xi\in\Rd.
\end{equation*}

Given $(a,u)\in\mathcal{O}_{M}\times\mathcal{S}(\Rd,\Rd)$ it is well-known that $au\in\mathcal{S}(\Rd,\Rd)$. Thus we may define the \emph{Fourier multiplication operator} (or \emph{Fourier multiplier} for short) induced\footnote{Of course the Fourier transform in $\mathcal{S}(\Rd,\Rd)$ is defined componentwise.} by $a$ as
\begin{equation*}
  a(D)u:= \mathcal{F}^{-1}(a\,\hat u),\quad u\in \mathcal{S}(\Rd,\Rd).
\end{equation*}
Then, we have
\begin{equation*}
  a(D)\in\mathcal{L}(\mathcal{S}(\Rd,\Rd),\mathcal{S}(\Rd,\Rd))\cap\mathcal{L}(\mathcal{S}^{\prime}(\Rd,\Rd),\mathcal{S}^{\prime}(\Rd,\Rd)),
\end{equation*}
where $\mathcal{S}^{\prime}$ denotes the topological dual space of $\mathcal{S}$, i.e. the \emph{tempered $\Rd$-valued distributions on $\Rd$}. The convolution theorem for the Fourier transform implies that, given $u\in\mathcal{S}$, we have
\begin{equation}\label{eq:conv}
  a(D)u= (\mathcal{F}^{-1}a)\ast u\quad\text{in}\quad \mathcal{S}^{\prime}.
\end{equation}

\begin{rem}
We are rather interested in Fourier multipliers which extend to bounded operators on $L^{2}$ (and on normed subspaces of it) than on $\mathcal{S}$ or on $\mathcal{S}^{\prime}$. Therefore we use \eqref{eq:conv} to extend the admissible functions from $\mathcal{O}_{M}$ to $\mathcal{L}(\CC^{d})$-valued tempered distributions on $\Rd$, i.e. to
\begin{equation*}
  \mathcal{S}^{\prime}(\Rd,\mathcal{L}(\CC^{d})):=\mathcal{L}(\mathcal{S}(\Rd,\Rd),\mathcal{L}(\CC^{d})).
\end{equation*}
In fact, given $(a,u)\in \mathcal{S}^{\prime}(\Rd,\mathcal{L}(\CC^{d}))\times\mathcal{S}(\Rd,\Rd)$, it is well-known that the convolution $a\ast u$ is a well-defined element in $\mathcal{S}^{\prime}(\Rd,\Rd)$. Thus we may define
\begin{equation*}
  a(D)u:=\mathcal{F}^{-1}(a \hat u)=(\mathcal{F}^{-1}a)\ast u,\quad u\in \mathcal{S}(\Rd,\Rd),
\end{equation*}
by the convolution theorem. It is clear that
\begin{equation*}
  a(D)\,:\, \mathcal{S}(\Rd,\Rd)\to \mathcal{S}^{\prime}(\Rd,\Rd)
\end{equation*}
is a linear operator. Again we call $a(D)$ a \emph{Fourier multiplier} with symbol $a$.
\end{rem}

Obviously any Fourier multiplier is translation invariant. Conversely, a well-known result characterizes all translation invariant and \emph{bounded} operators on $L^{2}$ as Fourier multipliers with symbols in $L^{\infty}(\Rd,\mathcal{L}(\CC^{d}))$, cf.~\cite{Hoe1960,Gra2009}. We use this latter result to describe \emph{bounded translation invariant operators} on the Lie algebra $H^{\infty}(\Rd,\Rd)$.

\begin{lem}\label{lem:Fourier-multiplier}
Let $A$ be a \emph{continuous} linear operator on the Fr\'{e}chet space $H^\infty(\Rd,\Rd)$. Then the following three conditions are equivalent:
\begin{enumerate}
  \item $A$ commutes with any translation $\tau_{u}$, where $u\in \Rd$.
  \item $A$ commutes with $\nabla_{u}$ for each $u\in \Rd$ (constant vector field).
  \item There is an element $a\in \mathcal{S}^{\prime}(\Rd,L(\CC^{d}))$ such that $A = a(D)$.
\end{enumerate}
\end{lem}

\begin{proof}
(i) The equivalence of (1) and (2) is easy to verify and since we already remarked that Fourier multipliers are translation invariant, it suffices to verify that (1) implies (3).

(ii) Assume that $A$ is linear and continuous on $H^\infty(\Rd,\Rd)$ which commutes with translations.  The topology of $H^\infty(\Rd,\Rd)$ is induced by the family of semi-norms
\begin{equation*}
  p_m(u) := \norm{(1-\Delta)^m u}_{L^{2}}, \qquad m\in\mathbb{N}.
\end{equation*}
Note that Plancherel's theorem yields $\Vert(1-\Delta)^{-l}\Vert_{\mathcal{L}(L^2,L^2)}\le 1$ for all $l\ge 0$. Thus the family $(p_m)_{m\in\mathbb{N}}$ is ordered. Consequently
there exists $m_0\in\mathbb{N}$ and $C>0$ such that
\begin{equation}\label{eq:stack}
  \norm{Au}_{L^{2}}\le C p_{m_0}(u),\qquad u\in H^\infty(\Rd,\Rd).
\end{equation}
Let
\begin{equation*}
  A_{m_0} := A\circ(1-\Delta)^{-m_0},
\end{equation*}
Then, \eqref{eq:stack} implies that $A_{m_0}$ maps the space $H^\infty(\Rd,\Rd)$ continuously into $L^{2}(\Rd,\Rd)$. Thus, there is a unique bounded extension $\tilde{A}_{m_0}$ of $A_{m_0}$
to $L^{2}(\Rd,\Rd)$. By construction $\tilde{A}_{m_0}$ is translation invariant. Thus by the classical $L^{2}$-result there is some $a_0\in L^{\infty}(\Rd,\mathcal{L}(\CC^{d}))$ such that
\begin{equation*}
  \tilde{A}_{m_0} = \mathcal{F}^{-1}a_0\mathcal{F}.
\end{equation*}
Given $u\in H^\infty(\Rd,\Rd)$, we have
\begin{equation*}
  Au=\tilde{A}_{m_0}\circ(1-\Delta)^{m_0}(u) = \mathcal{F}^{-1}(a_0(\xi)(1+\abs{\xi}^{2})^{m_0} \hat u),
\end{equation*}
meaning that $A$ is a Fourier multiplier with symbol $(1+\abs{\xi}^{2})^{m_0}a_0(\xi)$, which clearly belongs to $\mathcal{S}^{\prime}(\Rd,\mathcal{L}(\CC^{d}))$.
\end{proof}

Given $r \in \mathbb{R}$, a Fourier multiplier $a(D)$ with symbol $a$ belonging
to $L^1_{loc}(\Rd,\mathcal{L}(\CC^{d}))$ is said to be of class $M^{r}(\Rd)$ iff
\begin{equation*}
  \norm{a(\xi)}\lesssim \left( 1 + \abs{\xi}^{2}\right)^{r/2},\quad a.e.
\end{equation*}
In this case we call $r$ the \emph{order} of $a(D)$.

\begin{rem}
By Lemma~\ref{lem:Fourier-multiplier}, any bounded Fourier operator on $H^\infty(\Rd,\Rd)$ has a finite non-negative order. In particular there are no bounded Fourier multipliers of ``infinite order'' on $H^\infty(\Rd,\Rd)$.
\end{rem}

\section{Elliptic Fourier multipliers}
\label{sec:elliptic-Fourier-multipliers}

\begin{defn}
A Fourier multiplier $a(D)$ with symbol $a\in M^r(\Rd)$ is called \emph{elliptic} iff $a(\xi)\in\mathcal{GL}(\CC^{d})$ for almost all $\xi\in\Rd$ and
\begin{equation*}
\norm{[a(\xi)]^{-1}}\lesssim \left( 1 + \abs{\xi}^{2}\right)^{-r/2},\quad a.e.
\end{equation*}
\end{defn}

\begin{rems}\label{rem:FMell}
(a) If a bounded translation invariant operator $A$ on the space $H^\infty(\Rd,\Rd)$ extends to a \emph{bounded isomorphism} from $H^q(\Rd,\Rd)$ to $H^{q-r}(\Rd,\Rd)$ for some $r\ge 0$ and $q\ge r$ then its order is $r$ and it is elliptic. Indeed, given $r\ge 0$ and $q\ge r$ and invoking Lemma~\ref{lem:Fourier-multiplier}, we know that there is $r_0\in \NN$ and $a\in M^{r_0}(\Rd)$ such that $A=a(D)$. The fact that $A$ extends to a bounded isomorphism from $H^q(\Rd,\Rd)$ to $H^{q-r}(\Rd,\Rd)$ and a similar argument as in the second part of Lemma~\ref{lem:Fourier-multiplier} imply that $a$ belongs to $M^r(\Rd)$ and that $a(D)$ is elliptic.

(b) Given $r\in\mathbb{R}$, the operator
\begin{equation*}
  \mathrm{diag}[(1-\Delta)^{r/2},\dotsc,(1-\Delta)^{r/2}]
\end{equation*}
is elliptic. Note that the operator $-\Delta$ is not elliptic in the sense of the definition given above. We treat Fourier multipliers of this kind below with a slightly different notion of ellipticity.

(c) Let $f$ be a bounded and smooth function on $\mathbb{R}$ with
\begin{equation*}
  \liminf_{\xi \to -\infty} f(\xi)=0.
\end{equation*}
Given $r\ge 0$, consider $a(\xi):=f(\xi)(1+\xi^{2})^{r/2}$, $\xi\in\mathbb{R}$. Then $a$ belongs to $M^r(\mathbb{R})$ and it is invertible. But it is not elliptic.

(d) Note that we consider a quite simple class of elliptic systems. In fact there is are more elaborated notions of ellipticity for systems, e.g. in the sense of Douglis--Nirenberg. These more general constructions allow e.g. to treat operators of the form $\mathrm{diag}\,[1-\Delta,(1-\Delta)^{2}]$ on $C^\infty(\mathbb{R}^{2},\mathbb{R}^{2})$. To keep the presentation simple we do not expand this branch here.
\end{rems}

It is an easy consequence of Plancherel's theorem that Fourier multipliers are bounded on the corresponding Sobolev spaces into $L^{2}$. A corresponding result is true for the inverse of an elliptic Fourier multiplier.

\begin{prop}\label{prop:elliptic}
Let $A$ be a bounded translation invariant homomorphism on $H^\infty(\Rd,\Rd)$. Then, $A$ extends to a \emph{bounded isomorphism} from $H^q(\Rd,\Rd)$ onto $H^{q-r}(\Rd,\Rd)$ for some $r\ge 0$ and $q\ge r$ iff $A$ is an elliptic Fourier multiplier of order $r$, i.e. $A=a(D)$ with an elliptic symbol $a\in M^{r}(\Rd)$. In that case, $A$ extends to a bounded isomorphism from $H^q(\Rd,\Rd)$ onto $H^{q-r}(\Rd,\Rd)$ for any $r\ge 0$ and $q\ge r$.
\end{prop}

We next specify conditions on the symbol, guaranteeing that the corresponding Fourier multiplier is elliptic.

\begin{rems}
(a) Let $r\ge 0$ be given, and assume that $a_\pi\in L^\infty(\Rd,\mathcal{L}(\CC^{d}))$ is (positively) \emph{homogeneous of degree $r$}, i.e.
\begin{equation*}
  a_\pi(\lambda \xi)=\lambda^r\,a_\pi(\xi)\quad\text{in}\quad\mathcal{GL}(\CC^{d})
\end{equation*}
for all $\lambda\ge 0$ and almost all $\xi\in\Rd$. Clearly we have $a_\pi\in M^r(\Rd)$ in this situation. Note also that
\begin{equation}\label{eq:homogen}
  \lambda + a_\pi(\xi)=(\lambda^{2/r}+\abs{\xi}^{2})^{r/2}(\lambda_0+a_\pi(\xi_0)),\quad (\lambda, \xi)\in (0,\infty)\times\Rd,
\end{equation}
where
\begin{equation*}
  \lambda_0 := \frac{\lambda}{(\lambda^{2/r}+\abs{\xi}^{2})^{r/2}},\quad \xi_0 := \frac{\xi}{(\lambda^{2/r}+\abs{\xi}^{2})^{1/2}}.
\end{equation*}
Obviously we have that $\lambda_0^{2/r}+\vert\xi_0\vert^{2}=1$.

(b) Following~\cite{Ama1990a}, we call a homogeneous symbol $a_\pi$ \emph{normally elliptic} iff given $\xi\in\mathbb{S}^{d-1}$, all eigenvalues of $a_\pi(\xi)$ have positive real parts. In view of \eqref{eq:homogen}, normal ellipticity of $a_\pi$ implies that, given $(\lambda,\xi)\in (0,\infty)\times \Rd$, we have that $\lambda+a_\pi(\xi)$ and $\lambda_0+a_\pi(\xi_0)$ are invertible and  - recalling the homogeneity of $a_\pi$ - that
\begin{equation*}
  M := \sup_{(\mu,\eta)\in K_d} \norm{[\mu + a_\pi(\eta)]^{-1}}_{L(\CC^{d})},
\end{equation*}
is finite, where
\begin{equation*}
  K_{d} := \set{(\mu,\eta)\in (0,\infty)\times\Rd\,;\, \mu^{2/r}+\vert\eta\vert^{2}=1}.
\end{equation*}
Hence we get
\begin{equation}\label{eq:symbol-est}
 \norm{[\lambda + a_\pi(\xi)]^{-1}}_{L(\CC^{d})} \le M(\lambda^{2/r}+\abs{\xi}^{2})^{-r/2},
\end{equation}
provided $a_\pi$ is normally elliptic.
\end{rems}

We summarize the above considerations by noting the following result.

\begin{prop}\label{prop:pos.elliptic}
Let $\lambda>0$ and assume that $a_\pi\in M^r(\Rd)$ is a homogeneous symbol of degree $r$, which is normally elliptic. Then the corresponding Fourier multiplier $\lambda + a_\pi(D)$ is elliptic.
\end{prop}

The following result is a further consequence of \eqref{eq:symbol-est}.

\begin{cor}\label{cor:para-est}
Let $q\in \mathbb{R}$ and $\lambda \ge 1$ be given. Assume further that $a_\pi\in M^r(\Rd)$ is a homogeneous symbol  of degree $r$, which is normally elliptic. Then there is a $C_\ast>0$ such that
\begin{equation*}
 \lambda \norm{u}_{H^{q-r}} + \norm{u}_{H^q} \le C_\ast \norm{(\lambda + a_\pi(D))u}_{H^{q-r}}, \quad u\in H^q(\Rd,\Rd).
\end{equation*}
The constant $C_\ast$ can be chosen independently of $q$ and $\lambda\ge 1$.
\end{cor}

A Fourier multiplier $a(D)$ is said to be \emph{classical} if there is a homogeneous symbol $a_\pi$ of degree $r$ and a $r_0 < r$ such that $a-a_\pi\in M^{r_0}(\Rd)$. In this situation we call $a_\pi$ the \emph{principal symbol} of $a$. A classical Fourier multiplier is said to be \emph{normally elliptic}, iff its principal symbol is normally elliptic.

\begin{prop}\label{prop:ellip-classic}
Let $a(D)$ be a normally elliptic classical Fourier multiplier. Then there is a $\lambda_\ast>0$ such that $\lambda + a(D)$ is elliptic for any $\lambda\ge\lambda_\ast$.
\end{prop}

\begin{proof}
(a) Let $a(D)$ be a classical Fourier multiplier and denote by $a_\pi\in M^r(\Rd)$ its principle symbol. It suffices to show that, given $q\in\mathbb{R}$, there is a $\lambda_\ast>0$ and a $c_\ast>0$ such that
\begin{equation}\label{eq:coerc-classic}
  \norm{(\lambda + a(D))u}_{H^{q-r}} \ge c_\ast \norm{u}_{H^q},\quad u\in H^q(\Rd,\Rd),
\end{equation}
provided $\lambda\ge\lambda_\ast$. To do so, we shall apply the method of continuity, c.f. Theorem 5.2 in~\cite{GT1998}.

(b) By assumption there is a $r_0 < r$ such that $a_0:=a-a_\pi$ belongs to $S^{r_0}$. Thus there is a $C_0>0$ such that
\begin{equation}\label{lo-est}
  \norm{a_0(D)u}_{H^{q-r}} \le C_0 \norm{u}_{H^{q-(r-r_0)}}.
\end{equation}
Let $C_\ast>0$ be the constant appearing in Corollary \ref{cor:para-est}. Note that
\begin{equation*}
  q-r<q-(r-r_0)<q.
\end{equation*}
Thus, by interpolation and the weighted Young inequality there is a $C_{1}>0$ such that
\begin{equation}\label{lo-est1}
 C_0\norm{u}_{H^{q-(r-r_0)}} \le \frac{1}{2C_\ast} \norm{u}_{H^q} + C_{1} \norm{u}_{H^{q-r}}.
\end{equation}

(c) By Corollary~\ref{cor:para-est} there is a $C_\ast$ such that
\begin{equation}\label{eq:est-pric-part}
  \lambda \norm{u}_{H^{q-r}} + \norm{u}_{H^q}\le C_\ast \norm{(\lambda + a_\pi(D))u}_{H^{q-r}},\quad u\in H^q(\Rd,\Rd),
\end{equation}
for all $\lambda\ge 1$.
Combining \eqref{lo-est}-\eqref{eq:est-pric-part}, we find
\begin{equation*}
  \norm{(\lambda +a_\pi(D) + t a_0(D))u}_{H^{q-r}}\ge \frac{\lambda \norm{u}_{H^{q-r}}}{C_\ast} + \frac{\norm{u}_{H^q}}{C_\ast} - \frac{\norm{u}_{H^q}}{2C_\ast} - C_{1}\norm{u}_{H^{q-r}}
\end{equation*}
for all $t\in [0,1]$, $\lambda\ge 1$,  and $u\in H^q(\Rd,\Rd)$. Choosing
$\lambda_\ast = \max\{C_\ast C_{1},1\}$ and $c_\ast=1/2C_\ast$ we get \eqref{eq:coerc-classic} from Theorem 5.2 in~\cite{GT1998}.
\end{proof}

Following~\cite{Bro1961}, a classical Fourier multiplier $a\in M^r(\Rd)$ is said to be \emph{strongly elliptic} iff there is an $\alpha>0$ such that
\begin{equation*}
  \mathrm{Re}\, (a_\pi(\xi)\eta\cdot\eta)\ge\alpha\abs{\xi}^{r}\vert\eta\vert^{2},\quad (\xi,\eta)\in \Rd\times \mathbb{C}^{d}.
\end{equation*}

\begin{rems}\label{rem:ell}
(a) Let $\lambda\in \mathbb{C}$ and $(\xi,\eta)\in\Rd\times\mathbb{C}^{d}$ with $\text{Re}\,\lambda\le 0$ and $\abs{\xi}=\vert\eta\vert=1$ be given. Then
\begin{equation*}
  \abs{[\lambda - a_\pi(\xi)]\eta} \ge \mathrm{Re}\,([-\lambda + a_\pi(\xi)]\eta\cdot\eta)\ge -\mathrm{Re}\,\lambda + \alpha > 0,
\end{equation*}
showing that all eigenvalues of the principal symbol $a_\pi$ do have a positive real part. This proves that any strongly elliptic Fourier multiplier is normally elliptic.

(b) The converse of the above remark is not true. To see this, consider
\begin{equation*}
  a_t(D):=
    \begin{bmatrix}
      -\Delta & -t\Delta\\
      0 & -\Delta\\
    \end{bmatrix}\,
\end{equation*}
where $t\in \mathbb{R}$ is a free parameter. Then, $a_t(D)$ is normally elliptic for any choice of $t$, but it is only strongly elliptic if $\abs{t} < 2$.

(c) A paradigmatic class of Fourier multipliers which fit into the above described framework are differential operators on $\Rd$ of even order\footnote{Is is known that normally elliptic differential operators are automatically of even order, cf.~\cite{Ama1990a}.} and with constant coefficients. To be more specific, let $k\in\NN$ be given, and choose coefficients $a_\alpha\in \mathcal{L}(\CC^{d})$, where $\alpha\in \NN^{d}$ with $\abs{\alpha} \le 2k$. Consider the differential operator
\begin{equation*}
  A:=\sum_{\vert\alpha\vert \le 2k}a_\alpha(-i\partial)^\alpha.
\end{equation*}
Then $A=a(D)$, where $a(\xi):=\sum_{\vert\alpha\vert \le 2k}a_\alpha\xi^\alpha$ denotes its symbol. It is clear that $a$ belongs to $M^{2k}$ and that it is classical in the above sense. Thus the above results are applicable to $A$.
\end{rems}

\section{Derivatives of the conjugate of a Fourier multiplier}
\label{sec:Fourier-multipliers-derivatives}

Let $A=a(D)$ be a Fourier multiplier of class $\mathcal{S}^{r}$ and
\begin{equation*}
 A_{\varphi} := R_{\varphi^{-1}}A R_{\varphi},
\end{equation*}
where $\varphi \in \DiffRd$. It was shown in~\cite[Lemma 3.2]{EK2014} that
\begin{equation*}
    \partial^{n}_{\varphi}A_{\varphi} (v,\delta\varphi_{1}, \dotsc ,\delta\varphi_{n}) = R_{\varphi}A_{n}R_{\varphi}^{-1}(v,\delta\varphi_{1}, \dotsc ,\delta\varphi_{n}),
\end{equation*}
where
\begin{equation*}
  A_{n}:= \partial^{n}_{\id}A_{\varphi} \in \mathcal{L}^{n+1}(\HRd{\infty},\HRd{\infty})
\end{equation*}
is the $(n+1)$-linear operator defined inductively by $A_{0} = A$ and
\begin{multline}\label{eq:recurrence_relation}
    A_{n+1}(u_{0},u_{1}, \dotsc , u_{n+1}) = \nabla_{u_{n+1}} \left(A_{n}(u_{0}, u_{1}, \dotsc , u_{n}) \right)\\
    - \sum_{k=0}^{n} A_{n}(u_{0}, u_{1}, \dotsc ,\nabla_{u_{n+1}} u_{k}, \dotsc , u_{n}),
\end{multline}
where $\nabla$ is the canonical derivative on $\Rd$.

\begin{lem}\label{lem:multi-Fourier-multipliers}
Let $u_{0},u_{1}, \dotsc , u_{n} \in \HRd{\infty}$. Then, the Fourier transform of $A_{n}(u_{0},u_{1}, \dotsc , u_{n})$,
noted for short $\widehat{A_{n}}$, can be written as
\begin{equation}\label{eq:multi-Fourier-multipliers}
  \widehat{A_{n}}(\xi) = \int_{\xi_{0} + \dotsb + \xi_{n} = \xi} a_{n}(\xi_{0}, \dotsc ,\xi_{n}) \left[\hat{u}_{0}(\xi_{0}), \dotsc ,\hat{u}_{n}(\xi_{n}) \right]\, d\mu
\end{equation}
where $d\mu$ is the Lebesgue measure on the subspace $\xi_{0} + \dotsb + \xi_{n} = \xi$ of $(\RR^{d})^{n+1}$ and
\begin{equation*}
  a_{n} : (\RR^{d})^{n+1} \to \mathcal{L}^{n+1}(\CC^{d},\CC^{d})
\end{equation*}
is the $(n+1)$-linear map defined inductively by $a_{0} = a$ and

\begin{multline}\label{eq:multi-Fourier-symbols}
  a_{n+1}(\xi_{0}, \dotsc ,\xi_{n+1}) = \\
  (2i\pi) \sum_{k=0}^{n} \Big[ a_{n}(\xi_{0}, \dotsc ,\xi_{k} + \xi_{n+1},\dotsc ,\xi_{n}) -  a_{n}(\xi_{0}, \dotsc ,\xi_{n}) \Big] \otimes \xi_{k}^{\sharp} ,
\end{multline}
where $\xi^{\sharp}$ is the linear functional defined by $\xi^{\sharp}(X) := \xi \cdot X$.
\end{lem}

Before giving the proof of this result, we would like to point out that
\begin{equation}\label{eq:covariant-derivative-FT}
  \widehat{\nabla_{u}w}(\xi) = -2i\pi \int_{\xi_{1}+\xi_{2} = \xi} \big(\hat{u}(\xi_{1}) \cdot \xi_{2}\big)\hat{w}(\xi_{2})\, d\mu,
\end{equation}
for all $u,v \in\HRd{\infty}$.

\begin{proof}[Proof of Lemma~\ref{lem:multi-Fourier-multipliers}]
The proof is achieved by induction on $n$. For $n=1$, using~\eqref{eq:recurrence_relation} and~\eqref{eq:covariant-derivative-FT}, we get
\begin{multline*}
  \widehat{A_{1}}(\xi) = -2i\pi\int_{\xi_{0} + \xi_{1} = \xi} \big(\hat{u}_{1}(\xi_{1}) \cdot \xi_{0}\big)a(\xi_{0})[\hat{u}_{0}(\xi_{0})]\, d\mu \\
  + 2i\pi\int_{\xi_{0} + \xi_{1} = \xi} \big(\hat{u}_{1}(\xi_{1}) \cdot \xi_{0}\big)a(\xi)[\hat{u}_{0}(\xi_{0})]\, d\mu.
\end{multline*}
Therefore, \eqref{eq:multi-Fourier-multipliers} is true for $n=1$ with
\begin{equation*}
  a_{1}(\xi_{0},\xi_{1}) = 2i\pi \Big(a(\xi_{0}+\xi_{1}) - a(\xi_{0})\Big) \otimes \xi_{0}^{\sharp}.
\end{equation*}
Suppose now that~\eqref{eq:multi-Fourier-multipliers} is true for $n$. Using again~\eqref{eq:recurrence_relation} and~\eqref{eq:covariant-derivative-FT}, we get
\begin{multline*}
  \widehat{A_{n+1}}(\xi) = -2i\pi\int_{\xi_{n+1} + \xi^{\prime} = \xi} \big(\hat{u}_{n+1}(\xi_{n+1}) \cdot \xi^{\prime}\big)\widehat{A_{n}}(\xi^{\prime})\, d\mu \\
  -\sum_{k=0}^{n} \int_{\xi_{0} + \dotsb + \xi_{n} = \xi} a_{n}(\xi_{0}, \dotsc ,\xi_{n}) \left[\hat{u}_{0}(\xi_{0}), \dotsc ,\widehat{\nabla_{u_{n+1}} u_{k}}(\xi_{k}), \dotsc , \hat{u}_{n}(\xi_{n}) \right]\, d\mu.
\end{multline*}
By Fubini's theorem and the recurrence hypothesis, the first term in the right hand side can be written as
\begin{multline*}
  -2i\pi\int_{\xi_{0} + \dotsb + \xi_{n+1} = \xi} \big(\hat{u}_{n+1}(\xi_{n+1}) \cdot (\xi_{0} + \dotsb + \xi_{n})\big) \\
  a_{n}(\xi_{0}, \dotsc ,\xi_{n}) \left[\hat{u}_{0}(\xi_{0}), \dotsc ,\hat{u}_{n}(\xi_{n}) \right]\, d\mu,
\end{multline*}
while each term in the sum can be written as
\begin{multline*}
  -2i\pi\int_{\xi_{0} + \dotsb + \xi_{n+1} = \xi} \big(\hat{u}_{n+1}(\xi_{n+1}) \cdot \xi_{k} \big) \\
  a_{n}(\xi_{0}, \dotsc ,\xi_{k}+\xi_{n+1}, \dotsc \xi_{n}) \left[\hat{u}_{0}(\xi_{0}), \dotsc ,\hat{u}_{n}(\xi_{n}) \right]\, d\mu.
\end{multline*}
This shows that~\eqref{eq:multi-Fourier-multipliers} is still true for $\widehat{A_{n+1}}$, with
\begin{multline*}
  a_{n+1}(\xi_{0}, \dotsc ,\xi_{n+1})\left[X_{0}, \dotsc , X_{n+1} \right] = \\
  -2i\pi\big(X_{n+1} \cdot ( \xi_{0} + \dotsb +\xi_{n}) \big) a_{n}(\xi_{0}, \dotsc ,\xi_{n}) \left[X_{0}, \dotsc , X_{n} \right] \\
  + 2i\pi\sum_{k=0}^{n} \left( X_{n+1} \cdot \xi_{k} \right) a_{n}(\xi_{0}, \dotsc ,\xi_{k} + \xi_{n+1},\dotsc ,\xi_{n}) \left[X_{0}, \dotsc , X_{n} \right],
\end{multline*}
or in a more condensed form
\begin{multline*}
  a_{n+1}(\xi_{0}, \dotsc ,\xi_{n+1}) = \\
  (2i\pi) \sum_{k=0}^{n} \Big[ a_{n}(\xi_{0}, \dotsc ,\xi_{k} + \xi_{n+1},\dotsc ,\xi_{n}) -  a_{n}(\xi_{0}, \dotsc ,\xi_{n}) \Big] \otimes \xi_{k}^{\sharp} .
\end{multline*}
\end{proof}

\begin{lem}\label{lem:nth-derivative-estimate}
For each $n \in \NN$, there exists $C_{n} >0$ such that
\begin{equation}\label{eq:nth-derivative-estimate}
  \norm{a_{n}(\xi_{0}, \dotsc ,\xi_{n})} \le
  C_{n} \left(\prod_{k=0}^{n}\lambda_{1}(\xi_{k})\right)\left( \sum_{J \subset I_{n}} \lambda_{r-1}\big(\xi_{0} + \sum_{j \in J} \xi_{j}\big) \right),
\end{equation}
for all $\xi_{0}, \dotsc , \xi_{n} \in \RR$, where $I_{n}:=\set{1, \dotsc ,n}$ and $\lambda_{r}(\xi) := (1 + \abs{\xi}^{2})^{r/2}$.
\end{lem}

To prove Lemma~\ref{lem:nth-derivative-estimate}, we will need one further Lemma. Therefore we will introduce first a few useful notations.
Let $n\ge 1$, we set
\begin{equation*}
  t_{n}(\xi) := a(\xi)\otimes \overbrace{\xi^{\sharp} \otimes \dotsb \otimes \xi^{\sharp}}^{\text{$n$-times}}
\end{equation*}
which is an $(n+1)$-order tensor. Given $1 \le r \le n$ and $p_{1} < \dotsb < p_{r}$ in $\set{1,\dotsc , n}$ we define
\begin{equation*}
  t_{n}^{p_{1},\dotsc ,p_{r}}(\xi_{0}, \dotsc ,\xi_{r-1})(\xi) := a(\xi) \otimes \xi^{\sharp} \otimes \dotsb \otimes \underbrace{\xi_{0}^{\sharp}}_{p_{1}} \otimes \dotsb \otimes \underbrace{\xi_{r-1}^{\sharp}}_{p_{r}} \otimes \dotsb \otimes \xi^{\sharp},
\end{equation*}
as the tensor obtained from $t_{n}$ by \emph{freezing} the variable $\xi$ at position $p_{1}, \dotsc , p_{r}$ to $\xi_{0}, \dotsc ,\xi_{r-1}$.

\begin{expl}
\begin{equation*}
  t_{2}^{1}(\xi_{0})(\xi) = a(\xi) \otimes \xi_{0}^{\sharp} \otimes \xi^{\sharp},
\end{equation*}
and
\begin{equation*}
  t_{3}^{1,3}(\xi_{0},\xi_{1})(\xi) = a(\xi) \otimes \xi_{0}^{\sharp} \otimes \xi^{\sharp} \otimes \xi_{1}^{\sharp}.
\end{equation*}
\end{expl}

Let $I_{r,n} := \set{r, \dotsc , n}$. We define now the following $(n+1)$-order tensor
\begin{multline*}
  s_{n}^{p_{1},\dotsc ,p_{r}}(\xi_{0}, \dotsc ,\xi_{r-1})(\xi_{r}, \dotsc ,\xi_{n}) := \sum_{\sigma \in \mathfrak{S}_{r}} \epsilon(\sigma) \sum_{J \subset I_{r,n}} (-1)^{\abs{J}}
  \\
  t_{n}^{p_{1},\dotsc ,p_{r}}(\xi_{\sigma(0)}, \dotsc ,\xi_{\sigma(r-1)})(\xi_{0} + \dotsb + \xi_{r-1} + \sum_{j \in J} \xi_{j}),
\end{multline*}
where $\abs{J}$ denotes the cardinal of $J$, $\mathfrak{S}_{r}$ is the symmetric group of order $r$ and $\epsilon(\sigma)$ is the signature of the permutation $\sigma \in \mathfrak{S}_{r}$.

\begin{rem}
  Note that the expression
  \begin{equation*}
    s_{n}^{p_{1},\dotsc ,p_{r}}(\xi_{0}, \dotsc ,\xi_{r-1})(\xi_{r}, \dotsc ,\xi_{n})
  \end{equation*}
  is \emph{skew-symmetric} in the variables $\xi_{0}, \dotsc ,\xi_{r-1}$ and \emph{symmetric} in the variables $\xi_{r}, \dotsc ,\xi_{n}$.
\end{rem}

\begin{expl}
\begin{equation*}
  s_{1}^{1}(\xi_{0})(\xi_{1}) = \big(a(\xi_{0}) - a(\xi_{0} + \xi_{1})\big)\otimes \xi_{0}^{\sharp},
\end{equation*}
and
\begin{equation*}
  s_{2}^{1,2}(\xi_{0},\xi_{1})(\xi_{2}) = \big( a(\xi_{0} + \xi_{1}) - a(\xi_{0} + \xi_{1} + \xi_{2}) \big) \otimes  \big(\xi_{0}^{\sharp} \wedge \xi_{1}^{\sharp}\big).
\end{equation*}
\end{expl}

Finally, given a sequence $b_{n}(\xi_{0}, \dotsc ,\xi_{n})$ of $(n+1)$-order tensors, we define
\begin{multline*}
  \mathrm{Rec}(b_{n})(\xi_{0}, \dotsc ,\xi_{n+1}) :=
  \\
  \sum_{k=0}^{n} \Big[ b_{n}(\xi_{0}, \dotsc ,\xi_{k} + \xi_{n+1},\dotsc ,\xi_{n}) -  b_{n}(\xi_{0}, \dotsc ,\xi_{n}) \Big] \otimes \xi_{k}^{\sharp}.
\end{multline*}

\begin{rem}
Recall that the sequence $a_n$ satisfies
\begin{equation*}
 a_{n+1}=\mathrm{Rec}(a_{n})\,.
\end{equation*}
This will be important later, to prove Lemma~\ref{lem:nth-derivative-estimate}.
\end{rem}

\begin{lem}\label{lem:rec-sn}
  The sequence $s_n^{p_{1},\dotsc ,p_{r}}$ satisfies the relation
  \begin{multline*}
    \mathrm{Rec}\left( s_{n}^{p_{1},\dotsc ,p_{r}}\right) (\xi_{0}, \dotsc ,\xi_{n+1}) = - s_{n+1}^{p_{1},\dotsc ,p_{r}}(\xi_{0}, \dotsc ,\xi_{r-1})(\xi_{r}, \dotsc ,\xi_{n+1})
    \\
    - s_{n+1}^{p_{1},\dotsc ,p_{r}, n+1}(\xi_{0}, \dotsc ,\xi_{r-1},\xi_{n+1})(\xi_{r}, \dotsc ,\xi_{n}).
  \end{multline*}
\end{lem}

\begin{proof}
The expression $\mathrm{Rec}\left( s_{n}^{p_{1},\dotsc ,p_{r}}\right) (\xi_{0}, \dotsc ,\xi_{n+1})$ is the sum of two terms
\begin{multline*}
  R_{1} := \sum_{k=0}^{r-1} \Big[ s_{n}^{p_{1},\dotsc ,p_{r}}(\xi_{0}, \dotsc ,\xi_{k} + \xi_{n+1}, \dotsc ,\xi_{r-1})(\xi_{r}, \dotsc ,\xi_{n})
  \\
  -  s_{n}^{p_{1},\dotsc ,p_{r}}(\xi_{0}, \dotsc ,\xi_{r-1})(\xi_{r}, \dotsc ,\xi_{n}) \Big] \otimes \xi_{k}^{\sharp},
\end{multline*}
and
\begin{multline*}
  R_{2} := \sum_{k=r}^n \Big[ s_{n}^{p_{1},\dotsc ,p_{r}}(\xi_{0}, \dotsc ,\xi_{r-1})(\xi_{r}, \dotsc ,\xi_{k} + \xi_{n+1}, \dotsc ,\xi_{n})
  \\
  -  s_{n}^{p_{1},\dotsc ,p_{r}}(\xi_{0}, \dotsc ,\xi_{r-1})(\xi_{r}, \dotsc ,\xi_{n}) \Big] \otimes \xi_{k}^{\sharp}.
\end{multline*}
Expanding $s_{n}^{p_{1},\dotsc ,p_{r}}$, we get first
\begin{multline*}
  R_{1} = \sum_{\sigma \in \mathfrak{S}_{r}} \epsilon(\sigma) \sum_{J \subset I_{r,n}} (-1)^{\abs{J}} \sum_{k=0}^{r-1} \Big\{
  \\
  t_{n}^{p_{1},\dotsc ,p_{r}}(\xi_{\sigma(0)}, \dotsc , \underbrace{\xi_{k} + \xi_{n+1}}_{\sigma^{-1}(k)}, \dotsc ,\xi_{\sigma(r-1)})(\xi_{0} + \dotsb + \xi_{r-1} + \xi_{n+1} + \sum_{j \in J} \xi_{j})\otimes \xi_{k}^{\sharp}
  \\
  -  t_{n}^{p_{1},\dotsc ,p_{r}}(\xi_{\sigma(0)}, \dotsc ,\xi_{\sigma(r-1)})(\xi_{0} + \dotsb + \xi_{r-1} + \sum_{j \in J} \xi_{j}) \otimes \xi_{k}^{\sharp} \Big\},
  \\
\end{multline*}
and using the linearity of $t_{n}^{p_{1},\dotsc ,p_{r}}$ in the first $r$ variables, we have
\begin{multline*}
  R_{1} = \sum_{\sigma \in \mathfrak{S}_{r}} \epsilon(\sigma) \sum_{J \subset I_{r,n}} (-1)^{\abs{J}} \Big\{
  \\
  t_{n}^{p_{1},\dotsc ,p_{r}}(\xi_{\sigma(0)}, \dotsc ,\xi_{\sigma(r-1)})(\xi_{0} + \dotsb + \xi_{r-1} + \xi_{n+1} + \sum_{j \in J} \xi_{j}) \otimes (\sum_{k=0}^{r-1} \xi_{k}^{\sharp})
  \\
  -  t_{n}^{p_{1},\dotsc ,p_{r}}(\xi_{\sigma(0)}, \dotsc ,\xi_{\sigma(r-1)})(\xi_{0} + \dotsb + \xi_{r-1} + \sum_{j \in J} \xi_{j}) \otimes (\sum_{k=0}^{r-1} \xi_{k}^{\sharp})
  \\
  + \sum_{k=0}^{r-1} t_{n}^{p_{1},\dotsc ,p_{r}}(\xi_{\sigma(0)}, \dotsc , \underbrace{\xi_{n+1}}_{\sigma^{-1}(k)}, \dotsc ,\xi_{\sigma(r-1)})(\xi_{0} + \dotsb + \xi_{r-1} + \xi_{n+1} + \sum_{j \in J} \xi_{j})\otimes \xi_{k}^{\sharp} \Big\}.
\end{multline*}
For the second term $R_{2}$, we get
\begin{multline*}
  R_{2} = \sum_{\sigma \in \mathfrak{S}_{r}} \epsilon(\sigma) \sum_{J \subset I_{r,n}} (-1)^{\abs{J}} \sum_{k=r}^{n}\Big\{
  \\
  t_{n}^{p_{1},\dotsc ,p_{r}}(\xi_{\sigma(0)}, \dotsc ,\xi_{\sigma(r-1)})(\xi_{0} + \dotsb + \xi_{r-1} + \delta_{J}(k)\xi_{n+1} + \sum_{j \in J} \xi_{j}) \otimes  \xi_{k}^{\sharp}
  \\
  -  t_{n}^{p_{1},\dotsc ,p_{r}}(\xi_{\sigma(0)}, \dotsc ,\xi_{\sigma(r-1)})(\xi_{0} + \dotsb + \xi_{r-1} + \sum_{j \in J} \xi_{j}) \otimes  \xi_{k}^{\sharp}\Big\},
\end{multline*}
where $\delta_{J}$ is the characteristic function of $J$. Note that, for $J$ given, the only non-zero terms in the sum $\sum_{k=r}^{n}$ are those for which $k$ belongs to $J$. We get thus
\begin{multline*}
  R_{2} = \sum_{\sigma \in \mathfrak{S}_{r}} \epsilon(\sigma) \sum_{J \subset I_{r,n}} (-1)^{\abs{J}} \Big\{
  \\
  t_{n}^{p_{1},\dotsc ,p_{r}}(\xi_{\sigma(0)}, \dotsc ,\xi_{\sigma(r-1)})(\xi_{0} + \dotsb + \xi_{r-1} + \xi_{n+1} + \sum_{j \in J} \xi_{j}) \otimes (\sum_{j \in J} \xi_{j}^{\sharp})
  \\
  -  t_{n}^{p_{1},\dotsc ,p_{r}}(\xi_{\sigma(0)}, \dotsc ,\xi_{\sigma(r-1)})(\xi_{0} + \dotsb + \xi_{r-1} + \sum_{j \in J} \xi_{j}) \otimes (\sum_{j \in J} \xi_{j}^{\sharp})\Big\}.
\end{multline*}
Summing up the two expressions, we obtain
\begin{multline*}
  R_{1} + R_{2} = \sum_{\sigma \in \mathfrak{S}_{r}} \epsilon(\sigma) \sum_{J \subset I_{r,n}} (-1)^{\abs{J}} \Big\{
  \\
  t_{n+1}^{p_{1},\dotsc ,p_{r}}(\xi_{\sigma(0)}, \dotsc ,\xi_{\sigma(r-1)})(\xi_{0} + \dotsb + \xi_{r-1} + \xi_{n+1} + \sum_{j \in J} \xi_{j})
  \\
  -  t_{n+1}^{p_{1},\dotsc ,p_{r}}(\xi_{\sigma(0)}, \dotsc ,\xi_{\sigma(r-1)})(\xi_{0} + \dotsb + \xi_{r-1} + \sum_{j \in J} \xi_{j})
  \\
  + \sum_{k=0}^{r-1} t_{n+1}^{p_{1},\dotsc ,p_{r},n+1}(\xi_{\sigma(0)}, \dotsc , \underbrace{\xi_{n+1}}_{\sigma^{-1}(k)}, \dotsc ,\xi_{\sigma(r-1)},\xi_{k})(\xi_{0} + \dotsb + \xi_{r-1} + \xi_{n+1} + \sum_{j \in J} \xi_{j})
  \\
  - t_{n+1}^{p_{1},\dotsc ,p_{r},n+1}(\xi_{\sigma(0)}, \dotsc ,\xi_{\sigma(r-1)},\xi_{n+1})(\xi_{0} + \dotsb + \xi_{r-1} + \xi_{n+1} + \sum_{j \in J} \xi_{j})\Big\}.
\end{multline*}
Now
\begin{multline*}
  \sum_{J \subset I_{r,n}} (-1)^{\abs{J}} \Big\{
  t_{n+1}^{p_{1},\dotsc ,p_{r}}(\xi_{\sigma(0)}, \dotsc ,\xi_{\sigma(r-1)})(\xi_{0} + \dotsb + \xi_{r-1} + \xi_{n+1} + \sum_{j \in J} \xi_{j})
  \\
  -  t_{n+1}^{p_{1},\dotsc ,p_{r}}(\xi_{\sigma(0)}, \dotsc ,\xi_{\sigma(r-1)})(\xi_{0} + \dotsb + \xi_{r-1} + \sum_{j \in J} \xi_{j}) \Big\}
\end{multline*}
is equal to
\begin{multline*}
  - \sum_{J \subset I_{r,n+1}} (-1)^{\abs{J}} t_{n+1}^{p_{1},\dotsc ,p_{r}}(\xi_{\sigma(0)}, \dotsc ,\xi_{\sigma(r-1)})(\xi_{0} + \dotsb + \xi_{r-1} + \sum_{j \in J} \xi_{j}),
\end{multline*}
whereas
\begin{multline*}
  \sum_{\sigma \in \mathfrak{S}_{r}} \epsilon(\sigma) \sum_{J \subset I_{r,n}} (-1)^{\abs{J}} \Big\{
  \\
  \sum_{k=0}^{r-1} t_{n+1}^{p_{1},\dotsc ,p_{r},n+1}(\xi_{\sigma(0)}, \dotsc , \underbrace{\xi_{n+1}}_{\sigma^{-1}(k)}, \dotsc ,\xi_{\sigma(r-1)},\xi_{k})(\xi_{0} + \dotsb + \xi_{r-1} + \xi_{n+1} + \sum_{j \in J} \xi_{j})
  \\
  - t_{n+1}^{p_{1},\dotsc ,p_{r},n+1}(\xi_{\sigma(0)}, \dotsc ,\xi_{\sigma(r-1)},\xi_{n+1})(\xi_{0} + \dotsb + \xi_{r-1} + \xi_{n+1} + \sum_{j \in J} \xi_{j})\Big\}
\end{multline*}
is equal to
\begin{multline*}
  -\sum_{\sigma \in \mathfrak{S}_{r+1}} \epsilon(\sigma) \sum_{J \subset I_{r,n}} (-1)^{\abs{J}}
  \\
   t_{n+1}^{p_{1},\dotsc ,p_{r},n+1}(\xi_{\sigma(0)}, \dotsc ,\xi_{\sigma(r-1)},\xi_{\sigma(n+1)})(\xi_{0} + \dotsb + \xi_{r-1} + \xi_{n+1} + \sum_{j \in J} \xi_{j}),
\end{multline*}
because $\mathfrak{S}_{r+1}$, the permutation group of $\set{0, \dots , r-1, n+1}$ can be written as
\begin{equation*}
  \mathfrak{S}_{r+1} = \mathfrak{S}_{r} \cup (0,n+1)\mathfrak{S}_{r} \cup \dotsb \cup (r-1,n+1)\mathfrak{S}_{r}.
\end{equation*}
Therefore, we have finally
\begin{multline*}
  R_{1} + R_{2} = - s_{n+1}^{p_{1},\dotsc ,p_{r}}(\xi_{0}, \dotsc ,\xi_{r-1})(\xi_{r}, \dotsc ,\xi_{n+1})
  \\
  - s_{n+1}^{p_{1},\dotsc ,p_{r},n+1}(\xi_{0}, \dotsc ,\xi_{r-1},\xi_{n+1})(\xi_{r}, \dotsc ,\xi_{n}),
\end{multline*}
which achieves the proof.
\end{proof}

\begin{proof}[Proof of Lemma~\ref{lem:nth-derivative-estimate}]
For $n=1$, we have
\begin{equation*}
  a_{1}(\xi_{0},\xi_{1}) = (2i\pi) \Big(a(\xi_{0}+\xi_{1}) - a(\xi_{0})\Big)\otimes \xi_{0}^{\sharp} = (-2i\pi)s_{1}^{1}(\xi_{0})(\xi_{1}).
\end{equation*}
By the mean value theorem and the fact $a(D)$ is in the class $\mathfrak{S}^{r}$, we get
\begin{equation*}
  \norm{a(\xi_{0}+\xi_{1}) - a(\xi_{0})} \lesssim \abs{\xi_{1}} \max_{[\xi_{0},\xi_{0}+\xi_{1}]}\lambda_{r-1}(\xi).
\end{equation*}
But the function $\lambda_{r-1}$ has no maximum on $\Rd$, thus the supremum will be achieved at the boundary of the
interval $[\xi_{0},\xi_{0}+\xi_{1}]$ and hence
\begin{equation*}
  \norm{a(\xi_{0}+\xi_{1}) - a(\xi_{0})} \lesssim \abs{\xi_{1}} \Big(\lambda_{r-1}(\xi_{0}) + \lambda_{r-1}(\xi_{0}+\xi_{1})\Big).
\end{equation*}
Therefore, we have
\begin{equation*}
  \abs{a_{1}(\xi_{0},\xi_{1})} \lesssim \abs{\xi_{0}}\abs{\xi_{1}} \Big(\lambda_{r-1}(\xi_{0}) + \lambda_{r-1}(\xi_{0}+\xi_{1})\Big),
\end{equation*}
and we get~\eqref{eq:nth-derivative-estimate} since $\abs{\xi} \le \lambda_{1}(\xi)$.

For $n \ge 2$, Lemma~\ref{lem:rec-sn} and \eqref{eq:multi-Fourier-symbols} ensure that $a_{n}$ can be written as a linear combination of $s_{n}^{p_{1},\dotsc ,p_{r}}$. Now using iteratively the mean value theorem (see~\cite[Lemma A.6]{EK2014} for the details), it can be shown that each $s_{n}^{p_{1},\dotsc ,p_{r}}$ satisfies the estimate
\begin{multline*}
  \abs{s_{n}^{p_{1},\dotsc ,p_{r}}(\xi_{0}, \dotsc ,\xi_{r-1})(\xi_{r}, \dotsc ,\xi_{n})}
  \\
  \lesssim \left(\prod_{k=0}^{n}\lambda_{1}(\xi_{k})\right) \left( \sum_{J \subset I_{r,n}} \lambda_{r-1} \big(\xi_{0} + \dotsb + \xi_{r-1} + \sum_{j \in J} \xi_{j}\big) \right),
\end{multline*}
but
\begin{equation*}
  \sum_{J \subset I_{r,n}} \lambda_{r-1} \big(\xi_{0} + \dotsb + \xi_{r-1} + \sum_{j \in J} \xi_{j}\big) \le \sum_{J \subset \set{1, \dotsc , n}} \lambda_{r-1} \big(\xi_{0} + \sum_{j \in J} \xi_{j}\big),
\end{equation*}
which achieves the proof.
\end{proof}
\medskip

{\bf{Acknowledgement.}} The authors are grateful to Elmar Schrohe and J\"{o}rg Seiler for helpful discussions about various topics on translation invariant operators. It is also a pleasure to thank David Lannes for stimulating discussions concerning commutator estimates. Martin Bauer was supported by FWF project P24625.


\end{document}